\documentclass{article} \usepackage{etoolbox} \newtoggle{author} \toggletrue{author}

\iftoggle{author}{
    \usepackage[utf8]{inputenc}
    \usepackage[T1]{fontenc}
    \usepackage[english]{babel}

    \usepackage[section]{algorithm}
    \usepackage{amsmath}
    \usepackage{amsthm}
    \usepackage{geometry}
    \usepackage{hyperref}
    \usepackage[sort]{natbib}

    \usepackage{thmtools}
    \usepackage[capitalize]{cleveref}

    \crefname{equation}{}{}
    \crefname{figure}{Figure}{Figures}
    \crefname{section}{section}{sections}
    \crefname{subsection}{subsection}{subsections}

    \let\oldparagraph=\paragraph
    \renewcommand\paragraph[1]{\oldparagraph{#1.}}

    \numberwithin{equation}{section}
}{}

\usepackage{algpseudocode}
\usepackage{amssymb}
\usepackage{microtype}
\usepackage{pgfplots}
\usepackage{resmes}

\iftoggle{author}{
    \theoremstyle{definition}
    \newtheorem{definition}{Definition}[section]

    \theoremstyle{plain}
    \newtheorem{proposition}[definition]{Proposition}
    \newtheorem{theorem}[definition]{Theorem}
    \newtheorem{lemma}[definition]{Lemma}

    \theoremstyle{remark}
    \newtheorem{remark}[definition]{Remark}
}{
    \newsiamremark{remark}{Remark}
}

\iftoggle{author}{
    \title{Quantitative Stability and Numerical Resolution of the Moment Measure Problem\thanks{The research of Y.~A.~R.\ was supported by the Bergman Distinguished Visiting Professorship at Stanford University, a Simons Fellowship in Mathematics, and NSF grants DMS-1906370,2204347,2506872.}}

    \author{Guillaume Bonnet\thanks{CEREMADE, CNRS, Université Paris-Dauphine, Université PSL, 75016 Paris, France (\texttt{bonnet@ceremade.dauphine.fr}).} \and Yanir A.~Rubinstein\thanks{University of Maryland and Stanford University (\texttt{yanir@alum.mit.edu}).}}
}{
    \headers{Stability and Discretization for Moment Measures}{G.~Bonnet and Y.~A.~Rubinstein}

    \title{Quantitative Stability and Numerical Resolution of the Moment Measure Problem\thanks{Submitted to the editors DATE.
    \funding{The research of Y.~A.~R.\ was supported by the Bergman Distinguished Visiting Professorship at Stanford University, a Simons Fellowship in Mathematics, and NSF grants DMS-1906370,2204347,2506872.}}}

    \author{Guillaume Bonnet\thanks{CEREMADE, CNRS, Université Paris-Dauphine, Université PSL, 75016 Paris, France (\email{bonnet@ceremade.dauphine.fr}).} \and Yanir A.~Rubinstein\thanks{University of Maryland and Stanford University (\email{yanir@alum.mit.edu}).}}
}

\ifpdf
\hypersetup{
    pdftitle={Quantitative Stability and Numerical Resolution of the Moment Measure Problem},
    pdfauthor={G.~Bonnet and Y.~A.~Rubinstein}
}
\fi

\pgfplotsset{compat=1.18}

\definecolor{tab1}{HTML}{1f77b4}
\definecolor{tab2}{HTML}{ff7f0e}
\definecolor{tab3}{HTML}{2ca02c}
\definecolor{tab4}{HTML}{d62728}
\definecolor{tab5}{HTML}{9467bd}
\definecolor{tab6}{HTML}{8c564b}
\definecolor{tab7}{HTML}{e377c2}
\definecolor{tab8}{HTML}{7f7f7f}
\definecolor{tab9}{HTML}{bcbd22}
\definecolor{tab10}{HTML}{17becf}

\newcommand{\NN}{\mathbb{N}}
\newcommand{\RR}{\mathbb{R}}
\newcommand{\ZZ}{\mathbb{Z}}

\newcommand{\cE}{\mathcal{E}}
\newcommand{\cH}{\mathcal{H}}
\newcommand{\cI}{\mathcal{I}}

\newcommand{\<}{\langle}
\renewcommand{\>}{\rangle}

\newcommand{\cvxind}[1]{{\bf 1}_{#1}^\infty}

\let\epsilon\undefined
\let\phi\undefined

\DeclareMathOperator{\Conv}{Conv}
\DeclareMathOperator{\Lag}{Lag}
\DeclareMathOperator{\Lip}{Lip}
\DeclareMathOperator{\Vol}{Vol}
\DeclareMathOperator{\diam}{diam}
\DeclareMathOperator{\inter}{int}
\DeclareMathOperator{\mm}{mm}
\DeclareMathOperator{\supp}{supp}
\DeclareMathOperator{\spn}{span}

\begin{document}
\maketitle

\begin{abstract}
    The moment measure problem consists in finding a convex function $\psi$ whose moment measure, i.e., the pushforward by $\nabla \psi$ of the measure with density $e^{-\psi(\,\cdot\,)}$, is prescribed. It is highly non-linear and less understood than the related optimal transport problem. We establish a quantitative stability estimate for this problem. This estimate validates, as well as leads us to introduce, an approach to the numerical resolution of the moment measure problem inspired by semi-discrete optimal transport, consisting in approximating the prescribed measure by a finitely supported one. We describe a Newton method for solving the discrete problem thus obtained, and perform numerical experiments, studying the experimental rates of convergence of the approximation beyond the predictions of the stability estimate.
\end{abstract}

\iftoggle{author}{
    \tableofcontents
}{

}

\section{Introduction}
\label{sec:intro}

The \emph{moment measure} of a convex function $\psi \colon \RR^d \to \RR \cup \{\infty\}$ is defined as the pushforward of the measure $e^{-\psi(x)}\, d x$ by the map $\nabla \psi$,
\begin{equation*}
    \mm(\psi):=(\nabla \psi)_\# (e^{-\psi(x)}\, d x)
\end{equation*}
(noting that $\nabla\psi(x)$
is defined for Lebesgue a.e.\ $x$ by Rademacher's theorem). The \emph{moment measure problem} (MMP) consists of inverting $\mm$, or more precisely, of finding a convex function $\psi$ whose moment measure is a given (prescribed) positive, finite mass, measure $\mu$ on $\RR^d$,
\begin{equation}
    \label{eq:pushforward_equation}
    \mm(\psi)=\mu.
\end{equation}
When the measure $\mu$ admits a density $f$, i.e., $\mu = f(y)\, dy$, the MMP is a weak formulation of
\begin{equation}
    \label{eq:pde_formulation}
    \begin{cases}
        f(\nabla \psi(x)) \det \nabla^2 \psi(x)
        = e^{-\psi(x)}, & x\in \RR^d, \\
        \supp\,\mu \subset \overline {\nabla \psi(\RR^d)}.
    \end{cases}
\end{equation}
Indeed, \cref{eq:pushforward_equation} means that
\begin{equation}
    \label{eq:pushforwarddefEq}
    \int_{(\nabla \psi)^{-1}(E)} e^{-\psi(x)}\, d x = \int_E f(y)\, d y
\end{equation}
for Borel sets $E \subset \RR^d$; assuming that $\psi$ is $C^2$ and strongly convex, $\nabla\psi$ is one-to-one, and so setting $y=\nabla\psi(x)$ implies \cref{eq:pde_formulation}.

As illustrated by the formulation \cref{eq:pushforward_equation}, $\mm$ is a highly non-linear second-order operator. In particular, it is not as well-understood as the closely related, classical, Monge--Ampère operator. Our aim in this article is twofold: (i) to study the first numerical method for the MMP, (ii) to establish a quantitative stability estimate for solutions to \cref{eq:pushforward_equation}. Of course, (i) hinges on (ii), since our goal is, naturally, to \emph{prove} that the numerical method converges to the solution.

In retrospect, it is in some sense a miracle of convexity that $\mm$ even possesses a decent pluripotential theory reminiscent of the one for the much simpler Monge--Ampère operator.

The study of the so-called \emph{second boundary value problem} \cref{eq:pde_formulation} originates in the problem of finding Kähler--Ricci
solitons on toric manifolds. There, $\psi$ is assumed to have linear growth so that $\supp\,\mu$ is a (bounded) convex body, (in fact, a polytope $P$), and $f$ is of the form $e^{\langle a,\,\cdot\,\rangle}$ for a constant vector $a\in \RR^d$
so that $\int_P fdy=0$, i.e., $\mu$ is centered. Under these assumptions, a foundational theorem of Wang--Zhu says that \cref{eq:pde_formulation} is
uniquely smoothly solvable~\cite{wang2004} (for an exposition see Donaldson~\cite{donaldson2008}). Berman--Berndtsson considerably generalized this result by allowing arbitrary convex bodies and right-hand sides~\cite{berman2013} via a variational formulation inspired by optimal transport. The weak formulation \cref{eq:pushforward_equation} of the MMP is due to Cordero-Erausquin--Klartag~\cite{cordero2015}, and Santambrogio~\cite{santambrogio2016} further recast the MMP using a dual variational problem using tools from optimal transport.

It is important to distinguish the earlier results on the MMP from the later ones: neither is stronger nor more general than the other, since the earlier results study \emph{classical} solutions, hence establish higher-derivative estimates, while the later results are concerned with \emph{weak} solutions, hence consider less regular (i.e., more general) right-hand sides.

Since we will ultimately be interested in a numerical scheme, e.g., will allow $\mu$ to be an empirical measure, it will be natural for us to work in the latter framework of weak solutions.

The MMP features an inherent non-uniqueness with respect to translation in $\RR^d$ (this is perhaps easiest to see from \cref{eq:pde_formulation}). In the Kähler context this is referred to as uniqueness modulo automorphisms and goes back to Tian--Zhu~\cite{tian2000}; in the toric setting these automorphisms come from a complex torus and precisely correspond to translations in $\RR^d$. A foundational existence and uniqueness result of Cordero-Erausquin--Klartag (\cref{thm:existence_uniqueness}) determines the precise functional spaces on which \cref{eq:pushforward_equation} is invertible (modulo such translations). For this reason we will denote by
\begin{equation*}
    \mm^{-1}(\mu)
\end{equation*}
an essentially convex function on $\RR^d$ (\cref{def:essential_continuity}) well-defined up to a translation.

\begin{definition}[{\cite[Definition~2]{cordero2015}}]
    \label{def:essential_continuity}
    A convex function $\psi \colon \RR^d \to \RR \cup \{\infty\}$ is \emph{essentially continuous} if it is lower semicontinuous and
    \begin{equation*}
        \cH^{d-1}\Big(\big\{x \in \partial \{\psi < \infty\} \mid \psi(x) < \infty\big\}\Big) = 0,
    \end{equation*}
    where $\cH^{d-1}$ denotes the $(d-1)$-dimensional Hausdorff measure on $\RR^d$.
\end{definition}

A prototypical non-example is the convex indicator function
\begin{equation*}
    \cvxind K(x) := \begin{cases}
        0 &\text{if } x \in K, \\
        \infty &\text{otherwise},
    \end{cases}
\end{equation*}
for a closed convex set $K\subsetneq\RR^d$ with nonempty interior, since then $\{x \in \partial \{\psi < \infty\} \mid \psi(x) < \infty\}=\partial K$.

\begin{theorem}[{\cite[Proposition~1 and Theorem~2]{cordero2015}}]
    \label{thm:existence_uniqueness}
    \leavevmode\\
    {\rm (Existence)} A measure $\mu$ on $\RR^d$ with positive finite mass is the moment measure of some essentially continuous convex function if and only if it is centered and not supported on a hyperplane. \\
    {\rm (Uniqueness)} For $\mu$ centered and not supported on a hyperplane, $\mm^{-1}(\mu)$ is uniquely determined---in the class of essentially continuous convex functions---by $\mu$ up to translation.
\end{theorem}

\Cref{thm:existence_uniqueness} does \emph{not} say that $\mm(\psi_1)=\mm(\psi_2) \implies \psi_1(\,\cdot\,)=\psi_2(\,\cdot\,+v)$ for some $v\in\RR^d$; this only holds if \emph{both} $\psi_1$ and $\psi_2$ are essentially continuous; other, not essentially continuous, solutions can exist. For instance, let $\mu$ be the measure supported on a centered convex set $K \subsetneq \RR^d$ with nonempty interior, and with density $e^{-\|\,\cdot\,\|^2/2}$ on $K$. In other words, $\mu$ is a measure with density $e^{-\|\,\cdot\,\|^2/2 - \cvxind K}$ on $\RR^d$, and in particular it is centered and not supported on a hyperplane. Then, in addition to the essentially continuous functions $\mm^{-1}(\mu)$ guaranteed by \cref{thm:existence_uniqueness}, also $\mm\left(\|\,\cdot\,\|^2/2 + \cvxind K\right)=\mu$, and the latter is convex but not essentially continuous.

In this article, given a centered finite positive measure $\mu$ on $\RR^d$ that is not supported on a hyperplane, we  denote by
\begin{equation}
    \label{eq:psimu}
    \psi_\mu \colon \RR^d \to \RR \cup \{\infty\}
\end{equation}
an essentially continuous convex function whose moment measure is $\mu$, i.e., $\psi_\mu\in \mm^{-1}(\mu)$, and
\begin{equation}
    \label{eq:mminvEq}
    \mm^{-1}(\mu)=\{\psi_\mu(\,\cdot\,+v) \mid v\in \RR^d\}\cong\RR^d.
\end{equation}
In this notation, the present article establishes:

\begin{itemize}
    \item The quantitative stability (modulo translations) of the convex function $\psi_\mu$ with respect to the measure $\mu$ (\cref{thm:stability}).
    \item Numerical resolution of the MMP, i.e., rigorous numerical approximation of $\psi_\mu$, when $\mu$ is prescribed (\cref{sec:semidiscrete_mmp}).
\end{itemize}

\subsection{Stability for the MMP}
\label{subsec:intro_stability}

The next result gives quantitative stability for the moment measure problem. Conceptually, this amounts to estimating the operator norm of $\mm^{-1}$ (mapping measures to functions), part of the problem being to choose the right norms.

\begin{theorem}
    \label{thm:stability}
    Let $\mu\not=\nu$ be centered finite measures on $\RR^d$ with equal positive mass, not supported on hyperplanes. Then there exists $v \in \RR^d$ such that, for any constants $R,r$ satisfying
    \begin{align*}
        R &\geq \frac{1}{\mu(\RR^d)} \int_{\RR^d} \|y\|\, d \mu(y), &
        0 &< r \leq \frac{1}{\mu(\RR^d)} \inf_{w \in S^{d-1}} \int_{\RR^d} |\<w, y\>|\, d \mu(y),
    \end{align*}
    one has (recalling \cref{eq:psimu})
    \begin{equation}
        \label{eq:stability}
        \|e^{-\psi_\mu(\,\cdot\,)} - e^{-\psi_{\nu}(\,\cdot\, - v)}\|_{L^1(\RR^d)} \lesssim_{d, R/r} \mu(\RR^d) \left[\frac{W_1(\mu,\nu)}{\mu(\RR^d) R} \log \left(1 + \frac{\mu(\RR^d) R}{W_1(\mu,\nu)}\right)\right]^{1/2},
    \end{equation}
    where $W_1$ denotes the $1$-Wasserstein distance \cref{eq:kantorovich_rubinstein}.
\end{theorem}

\paragraph{Notation}

Throughout, $\alpha \lesssim_m \beta$ if there exists a constant $C_m > 0$, depending only on $m$, such that $\alpha \leq C_m \beta$.

A few comments are in order.
\begin{itemize}
    \item The assumptions on $\mu$ and $\nu$ are due to \cref{thm:existence_uniqueness}.
    \item   The role of $v$ on the left-hand side of \cref{eq:stability} is explained by \cref{eq:mminvEq}. The exponentials compensate for the fact that $\psi_\mu$ and $\psi_\nu$ can have different growth at infinity. Thus, the left-hand side is a measure of the distance between $\mm^{-1}(\mu)$ and $\mm^{-1}(\nu)$.
    \item On the right-hand side, $W_1$ is a standard ``$L^1$ notion'' of distance between two measures with finite first moments and equal mass,
    \begin{equation*}
        W_1(\mu,\nu) := \inf_{\gamma \in \Pi(\mu,\nu)} \int_{\RR^d \times \RR^d} |x - y|\, d \gamma(x, y),
    \end{equation*}
    where $\Pi(\mu,\nu)$ denotes so-called admissible transport plans from $\mu$ to $\nu$, i.e., the set of measures on $\RR^d \times \RR^d$ with marginals $\mu$ and $\nu$. In our proof we will rather use the dual characterization of the $1$-Wasserstein distance (see, e.g.,~\cite{santambrogio2015,villani2009}):
    \begin{equation}
        \label{eq:kantorovich_rubinstein}
        W_1(\mu,\nu) = \inf \bigg\{\int_{\RR^d} \varphi\, d \mu - \int_{\RR^d} \varphi\, d \nu \,\mathop{\Big|}\, \varphi \colon \RR^d \to \RR \text{ $1$-Lipschitz}\bigg\}.
    \end{equation}
    \item The estimate \cref{eq:stability} is compatible with the following easily verified invariance properties of the MMP:
    \begin{align}
        \label{eq:invarianceMMPeq}
        \mm(\psi+C) &= e^{-C} \mm(\psi), &
        \mm(\psi(\,\cdot\, / \lambda)) &= \lambda^d \mm(\psi)(\lambda \,\cdot\,).
    \end{align}
\end{itemize}

In order to prove \cref{thm:stability}, we consider an intermediate problem in which we look for a convex function which is not necessarily essentially continuous, but whose support is included in some given \emph{compact} convex set $K$ with nonempty interior. This intermediate problem, which we call the \emph{$K$-moment measure problem ($K$-MMP)}, may be interpreted as replacing $\RR^d$ by $K$ in the PDE formulation \cref{eq:pde_formulation}. When $K$ is large, it can be interpreted as an approximation of the space $\RR^d$. In addition to making the proof of \cref{thm:stability} cleaner, we believe that the $K$-MMP is of independent interest~: for instance, it arises naturally in some approaches to the numerical resolution of the moment measure problem~\cite{ji2024thesis}. It has only been briefly discussed in the literature~\cite{machado2025,santambrogio2016}, so we give some details in \cref{subsec:compact_problem}.

In this approach, the proof of \cref{thm:stability} naturally boils down to establishing two estimates:
\begin{itemize}
    \item A quantitative stability estimate (\cref{prop:lipschitz_lipschitz_estimate}) for the $K$-MMP just discussed.
    \item An error estimate (\cref{prop:lipschitz_nonlipschitz_estimate}) between the $K$-MMP and the original MMP.
\end{itemize}

\subsection{Numerical resolution}
\label{subsec:intro_num}

We study a numerical method consisting in approximating $\psi_\mu$ by the function $\psi_\nu$, where $\nu$ is a finitely supported measure that should be chosen close to $\mu$ in the $1$-Wasserstein distance. Thanks to $\nu$ being finitely supported, the problem of computing $\psi_\nu$ becomes a discrete problem that can be solved computationally.

This numerical method can be interpreted as a rather direct translation of the idea of semidiscrete optimal transport~\cite{aurenhammer1998,kitagawa2019,
lindsey2017,merigot2011,oliker1989}  to the setting of moment measures. Its validity in this setting is justified by our stability estimate, \cref{thm:stability}.

In our experiments, we compute $\psi_\nu$ using a damped Newton method, that we describe in \cref{sec:semidiscrete_mmp}. We investigate the effect of the choice of $\nu$ on the convergence of this Newton method and on the approximation error between $\psi_\mu$ and $\psi_\nu$. We observe that, experimentally, the rates of convergence of the numerical method with respect to $W_1(\mu, \nu)$ are typically better than what is guaranteed by \cref{thm:stability}, and that the convergence can be improved by applying some specific rules for the construction of the empirical measure $\nu$.

\subsection{Relation to other work}
\label{subsec:intro_other_work}

The proof of \cref{thm:stability} makes use of the sharp quantitative stability estimate for the Prékopa--Leindler inequality recently obtained by Figalli--van Hintum--Tiba~\cite{figalli2025}. It is also inspired by quantitative stability estimates for the optimal transport problem, see e.g.~\cite{delalande2023}; especially, in both settings, the $1$-Wasserstein distance is made to appear using Kantorovich--Rubinstein duality. The two main differences between \cref{eq:stability} and similar estimates in the optimal transport setting, such as the one in~\cite{delalande2023}, are the weaker norm appearing on the left-hand side of \cref{eq:stability}, and the logarithm on the right-hand side. The former originates from the estimate in~\cite{figalli2025}, and can be explained by the fact that the energy functional associated to the moment measure problem is different from the one associated to the optimal transport problem. The logarithm appears because the Legendre transform of the potential $\psi_\mu$ in the moment measure problem is typically not Lipschitz continuous; our intermediary result \cref{prop:lipschitz_nonlipschitz_estimate} may be interpreted as an estimate of this lack of Lipschitz continuity.

Quantitative stability of the moment measure problem was also recently and independently studied by Machado--Ramos~\cite{machado2025}. A version of our \cref{prop:lipschitz_lipschitz_estimate} can also be found in~\cite{machado2025}, while, to our knowledge, \cref{prop:lipschitz_nonlipschitz_estimate}, and thus \cref{thm:stability}, are original. Some stability estimates for the moment measure problem in the standard, non-compact case are also proved in~\cite{machado2025}, but under different assumptions and with smaller exponents. Observe that~\cite{machado2025} also features estimates for the error between the Legendre transforms of $\psi_\mu$ and $\psi_\nu$, which is a topic that we do not discuss herein. Quantitative stability for a quadratically regularized version of the moment measure problem was also studied by Delalande--Farinelli in~\cite{delalande2026}.

The idea of adapting techniques from semidiscrete optimal transport to the numerical resolution of the moment measure problem is rather natural. Klartag--Mérigot--Santambrogio previously performed numerical experiments using this approach and presented results in some talks, e.g.,~\cite{klartag2017}.

Finite-difference discretization of the system \cref{eq:pde_formulation} was first considered in the thesis of Ji~\cite{ji2024thesis}. In a forthcoming article joint with Ji~\cite{bjr} we synthesize our approach here with his, and also pursue applications to geometry, such as the Ricci flow and K\"ahler--Ricci solitons.

It is worth pointing out that the (special) moment measure problems arising in Kähler geometry have attracted intense efforts in the mathematical physics community, as they describe Kähler--Ricci solitons that in turns can be used to construct Klebanov--Tseytlin like supergravity solutions, see, e.g., \cite{bunch2008,donaldson2009,doran2008,douglas2008,hall2016,headrick2005,headrick2007}. The numerical results in those works employ completely different methods, based mainly on quantization (as opposed to PDE, as in the present article). Moreover, some of these works have attempted to use such solitons to study the Ricci flow, that itself has applications to renormalization flow in the mathematical physics literature.

\subsection{Outline}
\label{subsec:intro_outline}

In \cref{sec:moment_measures}, after briefly reviewing some aspects of the theory of moment measures, we discuss the intermediary problem involving a compact approximation of $\RR^d$ that we use in the proof of \cref{thm:stability}. \Cref{sec:stability_estimate} is devoted to the proof of \cref{thm:stability}, and of the intermediary results \cref{prop:lipschitz_lipschitz_estimate} and \cref{prop:lipschitz_nonlipschitz_estimate}. In \cref{sec:semidiscrete_mmp}, we discuss the approximation of $\psi_{\mu}$ by $\psi_\nu$, where $\nu$ is a finitely supported measure, and a Newton method for computing $\psi_\nu$. We present some numerical results in \cref{sec:numerical_experiments}.

\section{Moment measure problems with Lipschitz constraints}
\label{sec:moment_measures}

In \cref{subsec:variational_formulation}, we briefly review the variational formulation of the moment measure problem. In \cref{subsec:compact_problem}, we study a variant of the problem involving an approximation of $\RR^d$ by a compact domain, which will be useful for our presentation of the proof of \cref{thm:stability} in \cref{sec:stability_estimate}.

\subsection{Variational formulation}
\label{subsec:variational_formulation}

For any function $\varphi \colon \RR^d \to \RR \cup \{\infty\}$, let
\begin{equation}
    \label{eq:cI}
    \cI(\varphi) := -\log \int_{\RR^d} e^{-\varphi^*(x)}\, d x,
\end{equation}
where
\begin{equation*}
    \varphi^* \colon x \mapsto \sup_{y\in\RR^d} (\<x,y\> - \varphi(y))
\end{equation*}
denotes the Legendre transform of $\varphi$. Observe that $\cI(\varphi)$ can take values $\infty$ and $-\infty$ (but is finite as soon as the measure $e^{-\varphi^*(x)}\, d x$ has finite positive mass). For any probability measure $\mu$ on $\RR^d$ and any $\mu$-integrable function $\varphi \colon \RR^d \to \RR \cup \{\infty\}$, let
\begin{equation}
    \label{eq:cE_mu}
    \cE_\mu(\varphi) := \int_{\RR^d} \varphi(y)\, d \mu(y) + \cI(\varphi).
\end{equation}
Observe that we only define the functional $\cE_\mu$ for measures $\mu$ that are probability measures. This ensures that $\cE_\mu$ is invariant under addition of a constant (using that adding a constant to $\varphi$ amounts to subtracting it from $\varphi^*$),
\begin{equation*}
    \cE_\mu(\varphi + c) = \cE_\mu(\varphi), \quad \forall c \in \RR.
\end{equation*}
If the probability measure $\mu$ is centered, then $\cE_\mu$ is also invariant under addition of a linear function (using that adding a linear function to $\varphi$ amounts to a translation of $\varphi^*$),
\begin{equation*}
    \cE_\mu(\varphi + \<v,\,\cdot\,\>) = \cE_\mu(\varphi), \quad \forall v \in \RR^d.
\end{equation*}

The functional $\cE_\mu$ is related to the MMP. This idea goes back to Berman--Berndtsson~\cite[subsection~2.6]{berman2013}. The following statement is not explicitly formulated by Cordero-Erausquin--Klartag but follows from the proof of the existence part of~\cite[Theorem~2]{cordero2015} and from~\cite[Remark~13]{cordero2015}. (In the derivation, note that by \cref{eq:pushforwarddefEq},
\begin{equation}
    \label{eq:totalmassepsiEq}
    \int_{\RR^d} e^{-\psi(x)}\, d x = \mm(\psi)(\RR^d).
\end{equation}
After easily excluding the case $\int_{\RR^d} e^{-\varphi^*(x)}\, d x \in \{0, \infty\}$, we may always assume, up to adding a constant to $\varphi$, that $\int_{\RR^d} e^{-\varphi^*(x)}\, d x = 1$. Then, the equivalence follows from op. cit.)

\begin{proposition}
    \label{prop:energy_minimization}
    Let $\mu$ be a centered probability measure on $\RR^d$, not supported on a hyperplane. Let $\varphi \colon \RR^d \to \RR \cup \{\infty\}$ be a lower semicontinuous convex function. The following are equivalent:
    \begin{itemize}
        \item $\varphi$ is $\mu$-integrable and
        minimizes $\cE_\mu$ over all $\mu$-integrable functions.
        \item $\varphi^*$ is essentially continuous (\cref{def:essential_continuity}), $0 < \int_{\RR^d} e^{-\varphi^*(x)}\, d x < +\infty$, and
        \begin{equation}
            \label{eq:mmphistarintEq}
            \mm(\varphi^*)=\left(\int_{\RR^d} e^{-\varphi^*(x)}\, d x\right) \mu.
        \end{equation}
    \end{itemize}
\end{proposition}

There may also exist minimizers of $\cE_\mu$ that are not lower semicontinuous or not convex (take, e.g., a function $\varphi$ such that $\varphi^{**}$ minimizes $\cE_\mu$ but $\varphi(y) > \varphi^{**}(y)$ at a point $y$ where $\mu$ has no atom); however, it is easily verified that for any such minimizer $\varphi$, the lower semicontinuous function $\varphi^{**}$ also minimizes $\cE_\mu$, since $\cE_\mu(\varphi^{**}) \leq \cE_\mu(\varphi)$.

\subsection{\texorpdfstring{The $K$-MMP}{The 𝐾-MMP}}
\label{subsec:compact_problem}

Given a compact convex set $K \subset \RR^d$ with nonempty interior, let us define the following variants of the functionals $\cI$ and $\cE_\mu$ from \cref{eq:cI} and \cref{eq:cE_mu}. For any function $\varphi \colon \RR^d \to \RR \cup \{\infty\}$, let
\begin{equation*}
    \cI_K(\varphi) := -\log \int_K e^{-\varphi^*(x)}\, d x.
\end{equation*}
For any probability measure $\mu$ on $\RR^d$ and any $\mu$-integrable function $\varphi \colon \RR^d \to \RR \cup \{\infty\}$, let
\begin{equation*}
    \cE_{\mu, K}(\varphi) := \int_{\RR^d} \varphi(y)\, d \mu(y) + \cI_K(\varphi).
\end{equation*}

One can consider the problem of minimizing the energy $\cE_{\mu, K}$ over $\mu$-integrable functions. Let us notice that this new problem is related to the minimization of $\cE_\mu$ under a Lipschitz constraint.

Let
\begin{equation*}
    \Lip_K(\RR^d) := \{\varphi \colon \RR^d \to \RR \mid \forall y_1, y_2 \in \RR^d,\, \varphi(y_2) - \varphi(y_1) \leq \sigma_K(y_2 - y_1)\},
\end{equation*}
where $\sigma_K$ is the support function of $K$, i.e.,
\begin{equation*}
    \sigma_K(y) := \sup_{x \in K} \<x, y\>.
\end{equation*}
Observe that
\begin{equation}
    \label{eq:cEKcEeqLipeq}
    \varphi \in \Lip_K(\RR^d) \implies \cE_{\mu, K}(\varphi) = \cE_\mu(\varphi),
\end{equation}
since then $\varphi^*(x) = \infty$ for all $x \in \RR^d \setminus K$. Moreover, for any $\mu$-integrable function $\varphi \colon \RR^d \to \RR \cup \{\infty\}$ such that $\cE_{\mu, K}(\varphi) < \infty$,
\begin{equation*}
    \varphi_K \colon y \mapsto \sup_{x \in K} (\<x, y\> - \varphi^*(x))\in\Lip_K(\RR^d)
\end{equation*}
(interpreted as the largest convex function satisfying $\varphi_K \leq \varphi$ and $\varphi_K \in \Lip_K(\RR^d)$), and
\begin{equation*}
    \cE_\mu(\varphi_K) = \cE_{\mu, K}(\varphi_K) \leq \cE_{\mu, K}(\varphi).
\end{equation*}
Thus:
\begin{proposition}
    \label{prop:lipschitz_minimization_equivalence}
    Let $K \subset \RR^d$ be a compact convex set with nonempty interior, and let $\mu$ be a probability measure on $\RR^d$. If a function of $\Lip_K(\RR^d)$ minimizes $\cE_\mu$ over $\Lip_K(\RR^d)$, then it also minimizes $\cE_{\mu, K}$ over all $\mu$-integrable functions with values in $\RR \cup \{\infty\}$.
\end{proposition}

Next, we turn to the existence of minimizers of $\cE_{\mu, K}$ over $\Lip_K(\RR^d)$.

To that end, we employ standard arguments, similar to the ones used, e.g., in~\cite{cordero2015} for the study of the minimization of $\cE_\mu$, with the difference that the analysis is simpler here due compactness of $K$. For instance, observe that we do not need to assume that the measure $\mu$ is centered, nor that it is not supported on a hyperplane.

\begin{proposition}
    \label{prop:lipschitz_minimizer}
    Let $K \subset \RR^d$ be a compact convex set with nonempty interior, and let $\mu$ be a probability measure on $\RR^d$ with finite first moments. There exists a convex function $\varphi_{\mu, K} \in \Lip_K(\RR^d)$ that satisfies $\int_{\RR^d} e^{-(\varphi_{\mu, K})^*(x)}\, d x = 1$ and that minimizes $\cE_\mu$ over $\Lip_K(\RR^d)$ (equivalently, that minimizes $\cE_{\mu, K}$ over all $\mu$-integrable functions with values in $\RR \cup \{\infty\}$).
\end{proposition}

\begin{proof}
    Let $(\varphi_n)_{n \geq 0}$ be a minimizing sequence for $\cE_\mu$ on $\Lip_K(\RR^d)$. Up to replacing $\varphi_n$ by $(\varphi_n)^{**}$ and observing that $\cE_\mu((\varphi_n)^{**}) \leq \cE_\mu(\varphi_n)$, we can assume that $\varphi_n$ is convex for all $n \geq 0$. Since $\cE_\mu$ is invariant under addition of a constant, we can assume that $\varphi_n(0) = 0$ for all $n \geq 0$. Up to extracting a subsequence, we can assume that $(\varphi_n)_{n \geq 0}$ converges locally uniformly to some convex function $\varphi \in \Lip_K(\RR^d)$ (intuitively, this is because the sub-gradients are uniformly controlled).

    For any $x \in \RR^d$,
    \begin{align*}
        \varphi^*(x)
        &= \sup_{y \in \RR^d} (\<x, y\> - \varphi(y))
        = \sup_{y \in \RR^d} (\<x, y\> - \limsup_{n \to \infty} \varphi_n(y)) \\
        &= \sup_{y \in \RR^d} \sup_{n \geq 0} \inf_{m \geq n} (\<x, y\> - \varphi_m(y))
        \leq \sup_{n \geq 0} \inf_{m \geq n} \sup_{y \in \RR^d} (\<x, y\> - \varphi_m(y)) \\
        &= \liminf_{n \to \infty}\, (\varphi_n)^*(x).
    \end{align*}
    Therefore,
    \begin{equation}
        \label{eq:ephistar1eq}
        e^{-\varphi^*(x)} \geq \limsup_{n \to \infty} e^{-(\varphi_n)^*(x)}.
    \end{equation}

    For any $n \geq 0$, one has $\varphi_n \leq \sigma_K$, where we used the definition of $\Lip_K(\RR^d)$ and the fact that $\varphi_n(0) = 0$. Since the Legendre transform reverses order, we deduce that $(\varphi_n)^* \geq (\sigma_K)^* = \cvxind K$. Thus, $e^{-(\varphi_n)^*(\,\cdot\,)} \leq e^{-\cvxind K(\,\cdot\,)}$, which is integrable. Using Fatou's lemma and \cref{eq:ephistar1eq},
    \begin{align*}
        \cI(\varphi)
        &= -\log \int_{\RR^d} e^{-\varphi^*(x)}\, d x
        \leq -\log \int_{\RR^d} \limsup_{n \to \infty} e^{-(\varphi_n)^*(x)}\, d x \\
        &\leq -\log \limsup_{n \to \infty} \int_{\RR^d} e^{-(\varphi_n)^*(x)}\, d x
        = \liminf_{n \to \infty} \cI(\varphi_n).
    \end{align*}
    On the other hand, again by definition of $\Lip_K(\RR^d)$ and using $\varphi_n(0) = 0$, one has $|\varphi_n| \leq \max\{\sigma_K, \sigma_K(-\,\cdot\,)\}$; the latter is $\mu$-integrable since $K$ is compact, so $\sigma_K$ has linear growth, and $\mu$ has finite first moments. Thus, by the dominated convergence theorem,
    \begin{equation*}
        \lim_{n \to \infty} \int_{\RR^d} \varphi_n(y)\, d \mu(y) = \int_{\RR^d} \varphi(y)\, d \mu(y).
    \end{equation*}
    Thus, $\cE_\mu(\varphi) \leq \lim_{n \to \infty} \cE_\mu(\varphi_n)$, hence $\varphi$ minimizes $\cE_\mu$ on $\Lip_K(\RR^d)$.

    Even though $\varphi$ minimizes $\cE_\mu$ on $\Lip_K(\RR^d)$, it does not necessarily hold that $\int_{\RR^d} e^{-\varphi^*(x)}\, d x = 1$. This can be remedied as follows. We know that $\int_{\RR^d} e^{-\varphi^*(x)\, d x} \leq \int_{\RR^d} e^{-\cvxind{K}(x)} d x=|K| < \infty$, and also that $\int_{\RR^d} e^{-\varphi^*(x)}\, d x > 0$, since otherwise $\cE_\mu(\varphi) = \infty$, which would contradict the inequality $\cE_\mu(\varphi) \leq \cE_\mu(\sigma_K) < \infty$. We conclude the proof by choosing $\varphi_{\mu, K}$ as $\varphi - \log \int_{\RR^d} e^{-\varphi^*(x)}\, d x$. Since $\cE_\mu$ is invariant under addition of a constant, $\varphi_{\mu, K}$ still minimizes $\cE_\mu$ over $\Lip_K(\RR^d)$, and the additive constant $-\log \int_{\RR^d} e^{-\varphi^*(x)}\, d x$ is chosen so that $\int_{\RR^d} e^{-(\varphi_{\mu, K})^*(x)}\, d x = 1$.
\end{proof}

Finally, let us mention some other properties of this modified moment measure problem. We only sketch the proofs, since we will not need the following results in the next sections.

\begin{proposition}
    Let $K \subset \RR^d$ be a compact convex set with nonempty interior, and let $\mu$ be a probability measure on $\RR^d$ with finite first moments. Then any two convex minimizers of $\cE_\mu$ over $\Lip_K(\RR^d)$ coincide up to addition of an affine function.
\end{proposition}

\begin{proof}
    This follows from the strict convexity up to addition of affine functions of the functional $\cE_\mu$ on the convex set of convex functions that are finite in a neighborhood of the origin, see~\cite[Remark~13]{cordero2015}. In particular, this may be understood as a consequence of \cref{prop:convexity_estimate} below.
\end{proof}

\begin{remark}
    This does not imply that any sum of a convex minimizer and an affine function is a minimizer: even when the measure $\mu$ is centered and thus $\cE_\mu$ is invariant under addition of affine functions, one also needs that this sum belongs to the space $\Lip_K(\RR^d)$ of admissible functions.
\end{remark}

\begin{proposition}
    \label{prop:compact_problem_moment_measure}
    Let $K \subset \RR^d$ be a compact convex set with nonempty interior, and let $\mu$ be a probability measure on $\RR^d$ with finite first moments. Let $\varphi \in \Lip_K(\RR^d)$ be a minimizer of $\cE_\mu$ over $\Lip_K(\RR^d)$ that satisfies $\int_{\RR^d} e^{-\varphi^*(x)}\, d x = 1$. Then $\mm(\varphi^*) = \mu$.
\end{proposition}

\begin{proof}
    We follow the proof of the existence part of~\cite[Theorem~2]{cordero2015}, and refer to it for more details. Up to replacing $\varphi$ by $\varphi^{**}$, we may assume that $\varphi$ is convex. By \cref{prop:lipschitz_minimization_equivalence}, $\varphi$ also minimizes $\cE_{\mu, K}$ over all $\mu$-integrable functions with values in $\RR \cup \{\infty\}$, and, since $\varphi \in \Lip_K(\RR^d)$, one has $\varphi^* = \infty$ outside $K$. Thus, by assumption,
    \begin{equation*}
        \int_{K} e^{-\varphi^*(x)}\, d x
        =\int_{\RR^d} e^{-\varphi^*(x)}\, d x = 1.
    \end{equation*}
    Then, for any $\mu$-integrable function $\varphi_1 \colon \RR^d \to \RR \cup \{\infty\}$,
    \begin{align*}
        \int_{\RR^d} \varphi(y)\, d \mu(y)
        &= \int_{\RR^d} \varphi(y)\, d \mu(y) - \log \int_K e^{-\varphi^*(x)}\, d x
        \\
        &=\cE_K(\varphi)\le
        \cE_K(\varphi_1)
        \leq \int_{\RR^d} \varphi_1(y)\, d \mu(y) - \log \int_K e^{-\varphi^*_1(x)}\, d x.
    \end{align*}
    Using Jensen's inequality (applicable
    as $e^{-\varphi^*}$ is a probability density),
    \begin{equation*}
        -\log \int_K e^{-\varphi^*_1(x)}\, d x
        = -\log \int_K e^{\varphi^*(x) - \psi_1(x)} e^{-\varphi^*(x)}\, d x
        \leq \int_K (\varphi^*_1(x) - \varphi^*(x)) e^{-\varphi^*(x)}\, d x.
    \end{equation*}
    The last integral can also be taken on $\RR^d$, and then
    \begin{equation*}
        \int_{\RR^d} \varphi(y)\, d \mu(y) + \int_{\RR^d} \varphi^*(x) e^{-\varphi^*(x)}\, d x \leq \int_{\RR^d} \varphi_1(y)\, d \mu(y) + \int_{\RR^d} \varphi^*_1(x) e^{-\varphi^*(x)}\, d x.
    \end{equation*}
    This means that $\varphi$ and $\varphi^*$ are Kantorovich potentials associated to the optimal transport problem between $\mu$ and $e^{-\varphi^*(x)}\, d x$, and thus that $(\nabla \varphi^*)_\# (e^{-\varphi^*(x)}\, d x) = \mu$.
\end{proof}

\begin{remark}
    While $\varphi^*$ has moment measure $\mu$ in the above proposition, it is \emph{not} necessarily essentially continuous. Therefore, \cref{thm:existence_uniqueness} does not imply that it coincides with $\psi_\mu$ up to addition of a linear function. Adapting the arguments from~\cite{cordero2015}, one could probably characterize the set of convex functions with moment measure $\mu$ whose Legendre transforms minimize $\cE_{\mu, K}$, but this is not needed in this article.
\end{remark}

\section{Stability estimate for the MMP}
\label{sec:stability_estimate}

This section is devoted to the proof of \cref{thm:stability}. In \cref{subsec:convexity_estimate}, we state a convexity estimate for the functional $\cI$, as a rather direct consequence of the quantitative stability of the Prékopa--Leindler inequality. In \cref{subsec:compact_stability_estimate}, we deduce a stability estimate for the variant of the moment measure problem on a compact domain that we described in \cref{subsec:compact_problem}. In \cref{subsec:error_estimate_intermediary_problem}, we study estimate the error introduced when approximating the original problem by this variant. We conclude the proof of \cref{thm:stability} in \cref{subsec:proof_main_estimate}.

\subsection{Convexity estimate}
\label{subsec:convexity_estimate}

We rely on the following result of Figalli--van Hintum--Tiba~\cite{figalli2025} on the quantitative stability of the Prékopa--Leindler inequality.

\begin{theorem}
    \label{thm:quantitative_prekopa_leindler}
    Let $g_1$, $g_2 \colon \RR^d \to \RR_+$ be integrable functions with positive integrals. Let $h \colon \RR^d \to \RR_+$ be a measurable function such that for all $x_1$, $x_2 \in \RR^d$, one has $h(\frac{x_1 + x_2}{2}) \geq g_1(x_1)^{1/2} g_2(x_2)^{1/2}$. Then there exists $v \in \RR^d$ such that
    \begin{align*}
        &\left\|\frac{g_1}{\int_{\RR^d} g_1(x)\, d x} - \frac{g_2(\,\cdot\, - v)}{\int_{\RR^d} g_2(x)\, d x}\right\|_{L^1(\RR^d)} \\
        &\lesssim_d \left(\frac{\int_{\RR^d} h(x)\, d x}{\left(\int_{\RR^d} g_1(x)\, d x\right)^{1/2} \left(\int_{\RR^d} g_2(x)\, d x\right)^{1/2}} - 1\right)^{1/2}.
    \end{align*}
\end{theorem}

\begin{proof}
    The classical Prékopa--Leindler inequality states that
    \begin{equation*}
        \frac{\int_{\RR^d} h(x)\, d x}{\left(\int_{\RR^d} g_1(x)\, d x\right)^{1/2} \left(\int_{\RR^d} g_2(x)\, d x\right)^{1/2}} \geq 1.
    \end{equation*}
    Then, applying~\cite[Corollary~1.7]{figalli2025} to the functions
    \begin{align*}
        &\frac{g_1}{\int_{\RR^d} g_1(x)\, d x}, &
        &\frac{g_2}{\int_{\RR^d} g_2(x)\, d x}, &
        &\frac{h}{\left(\int_{\RR^d} g_1(x)\, d x\right)^{1/2} \left(\int_{\RR^d} g_2(x)\, d x\right)^{1/2}}
    \end{align*}
    yields the announced result.
\end{proof}

Following the standard argument for the proof of the convexity of $\cI$, i.e., applying the Prékopa--Leindler inequality to the functions $e^{-(\varphi_1)^*(\,\cdot\,)}$, $e^{-(\varphi_2)^*(\,\cdot\,)}$, and $e^{-((\varphi_1+\varphi_2)/2)^*(\,\cdot\,)}$, but replacing the Prékopa--Leindler inequality by its quantitative version, we deduce the following convexity estimate.

\begin{proposition}
    \label{prop:convexity_estimate}
    Let $\varphi_1$, $\varphi_2 \colon \RR^d \to \RR \cup \{\infty\}$ be two functions satisfying $0 < \int_{\RR^d} e^{-(\varphi_i)^*(x)}\, d x < \infty$, $i \in \{1, 2\}$. There exists $v \in \RR^d$ such that
    \begin{equation*}
        \frac{\cI(\varphi_1) + \cI(\varphi_2)}{2} - \cI\left(\frac{\varphi_1 + \varphi_2}{2}\right) \gtrsim_d \left\|\frac{e^{-(\varphi_1)^*(\,\cdot\,)}}{\int_{\RR^d} e^{-(\varphi_1)^*(x)}\, d x} - \frac{e^{-(\varphi_2)^*(\,\cdot\, - v)}}{\int_{\RR^d} e^{-(\varphi_2)^*(x)}\, d x}\right\|_{L^1(\RR^d)}^2.
    \end{equation*}
\end{proposition}

\begin{proof}
    Denote
    \begin{equation*}
        \varphi := \frac{\varphi_0 + \varphi_1}{2}.
    \end{equation*}
    For any $x_1$, $x_2 \in \RR^d$, one has
    \begin{align*}
        \varphi^*\left(\frac{x_1 + x_2}{2}\right)
        &= \sup_{y \in \RR^d} \left(\frac{1}{2} \<x_1, y\> + \frac{1}{2} \<x_2, y\> - \frac{1}{2} \varphi_1(y) - \frac{1}{2} \varphi_2(y)\right) \\
        &\leq \sup_{y_1 \in \RR^d} \left(\frac{1}{2} \<x_1, y_1\> - \frac{1}{2} \varphi_1(y_1)\right) + \sup_{y_2 \in \RR^d} \left(\frac{1}{2} \<x_2, y_2\> - \frac{1}{2} \varphi_2(y_2)\right) \\
        &= \frac{1}{2} (\varphi_1)^*(x_1) + \frac{1}{2} (\varphi_2)^*(x_2),
    \end{align*}
    i.e.,
    \begin{equation*}
        e^{-\varphi^*(\frac{x_1 + x_2}{2})} \geq \left(e^{-(\varphi_1)^*(x_1)}\right)^{1/2} \left(e^{-(\varphi_2)^*(x_2)}\right)^{1/2}.
    \end{equation*}
    Applying \cref{thm:quantitative_prekopa_leindler} to $e^{-(\varphi_0)^*(\,\cdot\,)}$, $e^{-(\varphi_1)^*(\,\cdot\,)}$, and $e^{-\varphi^*(\,\cdot\,)}$, there exists $v \in \RR^d$ such that
    \begin{align*}
        &\left\|\frac{e^{-(\varphi_1)^*(\,\cdot\,)}}{\int_{\RR^d} e^{-(\varphi_1)^*(x)}\, d x} - \frac{e^{-(\varphi_2)^*(\,\cdot\, - v)}}{\int_{\RR^d} e^{-(\varphi_2)^*(x)}\, d x}\right\|_{L^1(\RR^d)}^2 \\
        &\lesssim_d \frac{\int_{\RR^d} e^{-\varphi^*(x)}\, d x}{\left(\int_{\RR^d} e^{-(\varphi_1)^*(x)}\, d x\right)^{1/2} \left(\int_{\RR^d} e^{-(\varphi_2)^*(x)}\, d x\right)^{1/2}} - 1.
    \end{align*}
    Rewrite the right-hand side as
    \begin{equation*}
        \exp \left(\frac{\cI(\varphi_1) + \cI(\varphi_2)}{2} - \cI(\varphi)\right) - 1,
    \end{equation*}
    which is less than $2 \left(\frac{\cI(\varphi_1) + \cI(\varphi_2)}{2} - \cI(\varphi)\right)$ for $\frac{\cI(\varphi_1) + \cI(\varphi_2)}{2} - \cI(\varphi)$ small enough. On the other hand, the announced result also holds for $\frac{\cI(\varphi_1) + \cI(\varphi_2)}{2} - \cI(\varphi)$ large, since one always has
    \begin{equation*}
        \left\|\frac{e^{-(\varphi_1)^*(\,\cdot\,)}}{\int_{\RR^d} e^{-(\varphi_1)^*(x)}\, d x} - \frac{e^{-(\varphi_2)^*(\,\cdot\, - v)}}{\int_{\RR^d} e^{-(\varphi_2)^*(x)}\, d x}\right\|_{L^1(\RR^d)}^2 \leq 4.
    \end{equation*}
    This concludes the proof.
\end{proof}

\subsection{\texorpdfstring{Stability estimate for the $K$-MMP}{Stability estimate for the 𝐾-MMP}}
\label{subsec:compact_stability_estimate}

Given a probability measure $\mu$ on $\RR^d$ with finite first moments and a compact convex set $K \subset \RR^d$ with nonempty interior, \cref{prop:lipschitz_minimizer} furnishes a function
\begin{equation*}
    \varphi_{\mu, K}
\end{equation*}
that minimizes $\cE_\mu$ over $\Lip_K(\RR^d)$ (equivalently, minimizes $\cE_{\mu, K}$ over all $\mu$-integrable functions with values in $\RR \cup \{\infty\}$). Set
\begin{equation}
    \label{eq:phimuK}
    \psi_{\mu, K} := (\varphi_{\mu, K})^*.
\end{equation}
Recall that
\begin{equation*}
    \int_{\RR^d} e^{-\psi_{\mu, K}(x)}\, d x = 1,
\end{equation*}
and that by \cref{prop:compact_problem_moment_measure},
\begin{equation}
    \label{eq:psimuK}
    \mm(\psi_{\mu, K})=\mu.
\end{equation}
This means that $\psi_{\mu, K}$ solve the $K$-MMP for $\mu$. We obtain the following stability estimate for the $K$-MMP under variation of the measure $\mu$.

\begin{proposition}
    \label{prop:lipschitz_lipschitz_estimate}
    Let $\mu$, $\nu$ be two probability measures on $\RR^d$ with finite first moments, and let $K \subset \RR^d$ be a compact convex set with nonempty interior. There exists $v \in \RR^d$ such that
    \begin{equation*}
        \|e^{-\psi_{\mu, K(\,\cdot\,)}} - e^{-\psi_{\nu, K}(\,\cdot\, - v)}\|_{L^1(\RR^d)} \lesssim_d (\diam(K) W_1(\mu, \nu))^{1/2}.
    \end{equation*}
\end{proposition}

\begin{proof}
    By \cref{prop:convexity_estimate}, and using that $\int_{\RR^d} e^{-\psi_{\mu, K}(x)}\, d x = \int_{\RR^d} e^{-\psi_{\nu, K}(x)}\, d x = 1$, there exists $v \in \RR^d$ such that
    \begin{equation*}
        \|e^{-\psi_{\mu, K(\,\cdot\,)}} - e^{-\psi_{\nu, K}(\,\cdot\, - v)}\|_{L^1(\RR^d)}^2 \lesssim_d \frac{\cI(\varphi_{\mu, K}) + \cI(\varphi_{\nu, K})}{2} - \cI\left(\frac{\varphi_{\mu, K} + \varphi_{\nu, K}}{2}\right).
    \end{equation*}
    Using that $\varphi_{\mu, K}$ minimizes $\cE_\mu$ on $\Lip_K(\RR^d)$ and that the latter
    is a convex space, one has $\cE_\mu(\varphi_{\mu, K}) \leq \cE_\mu(\frac{\varphi_{\mu, K} + \varphi_{\nu, K}}{2})$, i.e.,
    \begin{equation*}
        \int_{\RR^d} \varphi_{\mu, K}(y)\, d \mu(y) + \cI(\varphi_{\mu, K}) \leq \int_{\RR^d} \frac{\varphi_{\mu, K}(y) + \varphi_{\nu, K}(y)}{2}\, d \mu(y) + \cI\left(\frac{\varphi_{\mu, K} + \varphi_{\nu, K}}{2}\right),
    \end{equation*}
    i.e.,
    \begin{equation*}
        \cI(\varphi_{\mu, K}) - \cI\left(\frac{\varphi_{\mu, K} + \varphi_{\nu, K}}{2}\right) \leq \frac{1}{2} \int_{\RR^d} \varphi_{\nu, K}(y) - \varphi_{\mu, K}(y)\, d \mu(y).
    \end{equation*}
    Similarly,
    \begin{equation*}
        \cI(\varphi_{\nu, K}) - \cI\left(\frac{\varphi_{\mu, K} + \varphi_{\nu, K}}{2}\right) \leq \frac{1}{2} \int_{\RR^d} \varphi_{\mu, K}(y) - \varphi_{\nu, K}(y)\, d \nu(y).
    \end{equation*}
    Therefore,
    \begin{equation*}
        \frac{\cI(\varphi_{\mu, K}) + \cI(\varphi_{\nu, K})}{2} - \cI\left(\frac{\varphi_{\mu, K} + \varphi_{\nu, K}}{2}\right) \leq \frac{1}{4} \int_{\RR^d} \varphi_{\nu, K}(y) - \varphi_{\mu, K}(y)\, d (\mu - \nu)(y).
    \end{equation*}
    Since $\varphi_{\mu, K}$ and $\varphi_{\nu, K}$ belong to $\Lip_K(\RR^d)$, the function $\varphi_{\nu, K} - \varphi_{\mu, K}$ is $\diam(K)$-Lipschitz. Using the Kantorovich--Rubinstein duality formula \cref{eq:kantorovich_rubinstein}, we deduce that
    \begin{equation*}
        \frac{\cI(\varphi_{\mu, K}) + \cI(\varphi_{\nu, K})}{2} - \cI\left(\frac{\varphi_{\mu, K} + \varphi_{\nu, K}}{2}\right) \leq \frac{1}{4} \diam(K) W_1(\mu, \nu),
    \end{equation*}
    as claimed.
\end{proof}

\subsection{\texorpdfstring{Error estimate for the $K$-MMP}{Error estimate for the 𝐾-MMP}}
\label{subsec:error_estimate_intermediary_problem}

For a given $\mu$ we thus obtain solutions $\psi_{\mu,K}$ and $\psi_\mu$ of the $K$-MMP and the MMP, respectively. The next result makes precise sense of the intuition that the former should converge to the latter as $K$ increases to $\RR^d$ (this means that $L\to\infty$ below).

\begin{proposition}
    \label{prop:lipschitz_nonlipschitz_estimate}
    Let $\mu$ be a centered probability measure on $\RR^d$, not supported on a hyperplane, and let $K \subset \RR^d$ be a compact convex set with nonempty interior. Let $L > 0$ be such that $B_L \subset K$. There exists $v \in \RR^d$ such that, for all $R \geq \int_{\RR^d} \|y\|\, d \mu(y)$ and $0 < r \leq \inf_{w \in S^{d-1}} \int_{\RR^d} |\<w, y\>|\, d \mu(y)$,
    \begin{equation*}
        \|e^{-\psi_\mu(\,\cdot\,)} - e^{-\psi_{\mu, K}(\,\cdot\, - v)}\|_{L^1(\RR^d)} \lesssim_{d, R/r} \left(1 + (r L)^{(d-1)/2}\right) e^{-r L / (4 \sqrt{d})}
    \end{equation*}
    (recall \cref{eq:psimu} and \cref{eq:psimuK}).
\end{proposition}

The proof of \cref{prop:lipschitz_nonlipschitz_estimate} relies on \cref{prop:growth_estimate} that is proved in \cref{subsec:growth}. It will be convenient to use the notation
\begin{equation}
    \label{eq:phimu}
    \varphi_\mu:=\psi_\mu^*,
\end{equation}
and we continue to use the notation \cref{eq:phimuK}.

\begin{proof}[Proof of \cref{prop:lipschitz_nonlipschitz_estimate}]
    Using \cref{prop:convexity_estimate} and the fact that $\cE_\mu$ is the sum of $\cI$ and a linear term, there exists $v \in \RR^d$ such that (recall \cref{eq:phimu})
    \begin{equation*}
        \|e^{-\psi_\mu} - e^{-\psi_{\mu, K}(\,\cdot\, - v)}\|_{L^1(\RR^d)}^2 \lesssim_d \frac{\cE_\mu(\varphi_\mu) + \cE_\mu(\varphi_{\mu, K})}{2} - \cE_\mu\left(\frac{\varphi_\mu + \varphi_{\mu, K}}{2}\right).
    \end{equation*}
    Since $\varphi_\mu$ is a minimizer, $\cE_\mu(\varphi_\mu) \leq \cE_\mu(\frac{\varphi_\mu + \varphi_{\mu, K}}{2})$,
    \begin{equation*}
        \|e^{-\psi_\mu} - e^{-\psi_{\mu, K}(\,\cdot\, - v)}\|_{L^1(\RR^d)}^2 \lesssim_d \cE_\mu(\varphi_{\mu, K}) - \cE_\mu(\varphi_\mu).
    \end{equation*}
    Since $\varphi_{\mu, K}$ minimizes $\cE_\mu$ over $\Lip_K(\RR^d)$ (recall \cref{prop:lipschitz_minimizer}),
    \begin{equation*}
        \|e^{-\psi_\mu} - e^{-\psi_{\mu, K}(\,\cdot\, - v)}\|_{L^1(\RR^d)}^2 \lesssim_d \cE_\mu(\varphi) - \cE_\mu(\varphi_\mu),
        \quad \hbox{for all\ }
        \varphi\in\Lip_K(\RR^d).
        \end{equation*}
    We will build a suitable competitor $\varphi \in \Lip_K(\RR^d)$.

    By \cref{prop:growth_estimate}, there exists $v' \in \RR^d$ and a constant $C_d \in \RR$ depending only on $d$ such that $\psi_\mu(-v') < \infty$ and
    \begin{equation}
        \label{eq:lipchitz_nonlipschitz_estimate_proof_growth_estimate}
        \psi_\mu(x - v') \geq \frac{r}{2 \sqrt{d}} \|x\| - C_d - d \log R, \quad x \in \RR^d.
    \end{equation}
    We choose
    \begin{equation*}
        \varphi(y) := \sup_{x \in K} (\<x, y\> - \psi_\mu(x - v')).
    \end{equation*}
    It is verified that $\varphi \in \Lip_K(\RR^d)$, that $\varphi^*(x) = \psi_\mu(x - v')$ for all $x \in K$, and that $\varphi^*(x) = \infty$ for all $x \in \RR^d \setminus K$. For any $y \in \RR^d$, one has $\varphi(y) = \sup_{x \in K} (\<x, y\> - \psi_\mu(x - v')) \leq \sup_{x \in \RR^d} (\<x, y\> - \psi_\mu(x - v')) = \sup_{x \in K} (\<x + v', y\> - \psi_\mu(x)) = \varphi_\mu(y) + \<v', y\>$. Therefore,
    \begin{equation*}
        \int_{\RR^d} \varphi(y)\, d \mu(y)
        \leq \int_{\RR^d} [\varphi_\mu(y) + \<v', y\>]\, d \mu(y)
        = \int_{\RR^d} \varphi_\mu(y)\, d y,
    \end{equation*}
    and thus
    \begin{equation*}
        \cE_\mu(\varphi) - \cE_\mu(\varphi_\mu) \leq \cI(\varphi) - \cI(\varphi_\mu).
    \end{equation*}
    By \cref{eq:totalmassepsiEq}, as $\mu$ is a probability measure,
    $\int_{\RR^d} e^{-\psi_\mu(x)}\, d x = 1$, i.e., $\cI(\varphi_\mu) = 0$, so
    \begin{equation*}
        \cE_\mu(\varphi) - \cE_\mu(\varphi_\mu) \leq \cI(\varphi).
    \end{equation*}
    One has
    \begin{align*}
        \cI(\varphi)
        &= -\log \int_{\RR^d} e^{-\varphi^*(x)}\, d x
        = -\log \int_K e^{-\psi_\mu(x - v')}\, d x \\
        &= -\log \left(1 - \int_{\RR^d \setminus K} e^{-\psi_\mu(x - v')}\, d x\right)
        \leq -\log \left(1 - \int_{\RR^d \setminus B_L} e^{-\psi_\mu(x - v')}\, d x\right).
    \end{align*}
    Chaining the above estimates and using also that $\|e^{-\psi_\mu(\,\cdot\,)} - e^{-\psi_{\mu, K}(\,\cdot\, - v)}\|_{L^1(\RR^d)} \leq \|e^{-\psi_\mu(\,\cdot\,)}\|_{L^1(\RR^d)} + \|e^{-\psi_{\mu, K}(\,\cdot\, - v)}\|_{L^1(\RR^d)} = 2$, we deduce that
    \begin{align*}
        &\|e^{-\psi_\mu(\,\cdot\,)} - e^{-\psi_{\mu, K}(\,\cdot\, - v)}\|_{L^1(\RR^d)} \\
        &\lesssim_d \min \left\{1, \left(-\log \left(1 - \int_{\RR^d \setminus B_L} e^{-\psi_\mu(x - v')}\, d x\right)\right)^{1/2}\right\} \\
        &\lesssim \left(\int_{\RR^d \setminus B_L} e^{-\psi_\mu(x - v')}\, d x\right)^{1/2}.
    \end{align*}
    Using \cref{eq:lipchitz_nonlipschitz_estimate_proof_growth_estimate},
    \begin{equation*}
        \int_{\RR^d \setminus B_L} e^{-\psi_\mu(x - v')}\, d x \lesssim_d \left(\frac{R}{r}\right)^d \left(1 + (rL)^{d-1}\right) e^{-r L / (2 \sqrt{d})},
    \end{equation*}
    which concludes the proof.
\end{proof}

\subsection{A priori
growth estimate à la Wang--Zhu}
\label{subsec:growth}

The proof of \cref{prop:lipschitz_nonlipschitz_estimate} relied on an a priori growth estimate, \cref{prop:growth_estimate}. The statement of \cref{prop:growth_estimate} below is the analogue in the setting of weak solutions of the key a priori estimate in the setting of classical solutions (in the toric Kähler--Einstein literature) that itself essentially goes back to Wang--Zhu, see~\cite[subsection~3.4, Proposition~1]{donaldson2008}. Our argument is completely different from the original geometric approach and generalizes it.

\begin{proposition}
    \label{prop:growth_estimate}
    Let $\mu$ be a centered probability measure on $\RR^d$, not supported on a hyperplane. Let $R \geq \int_{\RR^d} \|y\|\, d \mu(y)$ and $0 < r \leq \inf_{w \in S^{d-1}} \int_{\RR^d} |\<w, y\>|\, d \mu(y)$. There exists $v \in \RR^d$ and a constant $C_d \in \RR$ depending only on $d$ such that $\psi_\mu(-v) < \infty$ and, for all $x \in \RR^d$,
    \begin{equation*}
        \psi_\mu(x - v) \geq \frac{r}{2 \sqrt{d}} \|x\| - C_d - d \log R.
    \end{equation*}
\end{proposition}

The proof makes use a measure-splitting trick, proved in \cref{lemma:separating_hyperplane} below.

\begin{proof}[Proof of \cref{prop:growth_estimate}]
    Let $(e_i)_{1 \leq i \leq d}$ denote the canonical basis of $\RR^d$. Let
    \begin{equation}
        \label{eq:wieq}
        w_i\in S^{d-1}, \quad \mu_{i, 1}+\mu_{i, 2}=\mu, \quad i\in\{1,\ldots,d\},
    \end{equation}
    be obtained by applying \cref{lemma:separating_hyperplane} $d$ times (with $e = e_1,\ldots,e_d$). Denoting
    $y_{i, 1}$ and $y_{i, 2}$ the barycenters of $\mu_{i, 1}$ and $\mu_{i, 2}$, we have
    \begin{equation}
        \label{eq:yi12eq}
        y_{i,1}, y_{i,2} \in \RR e_i, \quad i \in \{1,\ldots,d\}.
    \end{equation}
    It will be key to show that these barycenters are \emph{a priori} bounded \emph{away} from the origin (this is not a consequence of \cref{lemma:separating_hyperplane}).
    We will obtain that towards the end of this proof.

    First, note there exists $v \in \RR^d$ such that (recall \cref{eq:phimu}),
    \begin{equation}
        \label{eq:growth_estimate_proof_equalities}
        \varphi_\mu(y_{i, 1}) + \<v, y_{i, 1}\> = \varphi_\mu(y_{i, 2}) + \<v, y_{i, 2}\>,
        \quad i\in\{1,\ldots,d\}.
    \end{equation}
    Indeed, this amounts to solving the matrix equation
    \begin{equation*}
        \begin{pmatrix}
        y_{i, 2}-y_{i, 1}
        \end{pmatrix}_{1 \leq i \leq d}v
        =
        (\varphi_\mu(y_{i, 1})-\varphi_\mu(y_{i, 2}))_{1 \leq i \leq d},
    \end{equation*}
    where, by \cref{lemma:separating_hyperplane}, each row of the matrix on the left is a multiple of $e_i$. This equation is solvable since the $i$-th row is zero iff the $i$-th column on the right is zero.

    Using that $\mu_{i, 1}(\RR^d) + \mu_{i, 2}(\RR^d) = \mu(\RR^d)=1$, and \cref{eq:growth_estimate_proof_equalities}, one has, for $i \in \{1, \ldots, d\}$ and $j \in \{1, 2\}$,
    \begin{equation*}
        \varphi_\mu(y_{i, j}) + \<v, y_{i, j}\>
        = \mu_{i, 1}(\RR^d) (\varphi_\mu(y_{i, 1}) + \<v, y_{i, 1}\>) + \mu_{i, 2}(\RR^d) (\varphi_\mu(y_{i, 2}) + \<v, y_{i, 2}\>).
    \end{equation*}
    Next, convexity of $\varphi_\mu$ and Jensen's inequality give
    \begin{equation*}
        \varphi_\mu(y_{i, j})=\varphi_\mu\left(\frac{1}{\mu_{i, j}(\RR^d)} \int_{\RR^d}y\, d \mu_{i, j}(y)\right) \leq \frac{1}{\mu_{i, j}(\RR^d)} \int_{\RR^d} \varphi_\mu(y)\, d \mu_{i, j}(y).
    \end{equation*}
    Combined with $\mu_{i, 1} + \mu_{i, 2} = \mu$ and then that $\mu$ is centered, this implies
    \begin{equation}
        \label{eq:phimuijineq}
        \varphi_\mu(y_{i, j}) + \<v, y_{i, j}\>
        \leq \int_{\RR^d} \varphi_\mu(y) + \<v, y\>\, d \mu(y)
        = \int_{\RR^d} \varphi_\mu(y)\, d \mu(y)=\cE_\mu(\varphi_\mu),
    \end{equation}
    where the last equality holds since by \cref{eq:totalmassepsiEq},  $\int_{\RR^d} e^{-\psi_\mu(x)}\, d x = 1$  ($\mu$ is a probability measure), i.e., $\cI(\varphi_\mu) = 0$. Thus, \cref{eq:phimuijineq} and the fact that $\varphi_\mu$ is a minimizer of $\cE_\mu$ (by \cref{prop:energy_minimization}),
    \begin{equation*}
        \varphi_\mu(y_{i, j}) + \<v, y_{i, j}\>
        \leq \cE_\mu(\varphi_\mu)
        \leq \cE_\mu\left(\frac{1}{R} \|\cdot\|\right).
    \end{equation*}
    Using that the Legendre transform of $\frac{1}{R} \|\cdot\|$ is the convex indicator function of $\overline {B_{1/R}}$,
    \begin{equation}
        \label{eq:growth_estimate_proof_upper_bound}
        \varphi_\mu(y_{i, j}) + \<v, y_{i, j}\>
        \leq \frac{1}{R} \int_{\RR^d} \|y\|\, d \mu(y) - \log \Vol(B_{1/R}) \leq C_d + d \log R,
    \end{equation}
    where $C_d = 1 - \log |B_1|$ (remember that $|B_1|$ depends on the dimension $d$).

    For $1 \leq i \leq d$ and $1 \leq j \leq 2$, one has (recall \cref{eq:wieq} and that $\|w_i\|=1$)
    \begin{equation*}
        \|y_{i, j}\|
        \geq |\<w_i, y_{i, j}\>|
        \geq \mu_{i, j}(\RR^d) |\<w_i, y_{i, j}\>|.
    \end{equation*}
    The splitting properties of \cref{lemma:separating_hyperplane} imply
    \begin{align*}
        \mu_{i, 1}(\RR^d) |\<w_i, y_{i, 1}\>|
        + \mu_{i, 2}(\RR^d) |\<w_i, y_{i, 2}\>|
        &= \int_{\RR^d} |\<w_i, y\>|\, d \mu(y)
        \geq r, \\
        \mu_{i, 1}(\RR^d) |\<w_i, y_{i, 1}\>|
        - \mu_{i, 2}(\RR^d) |\<w_i, y_{i, 2}\>|
        &= \int_{\RR^d} \<w_i, y\>\, d \mu(y)
        = 0.
    \end{align*}
    Altogether,
    \begin{equation*}
        \|y_{i, j}\|
        \geq \mu_{i, j}(\RR^d) |\<w_i, y_{i, j}\>|
        \geq \frac{r}{2}.
    \end{equation*}
    By \cref{eq:yi12eq}, the points $y_{i, j}$ belong to the $d$ lines $\spn \{e_i\}$, with $\<e_i, y_{i, 1}\> \geq 0$ and $\<e_i, y_{i, 2}\> \leq 0$. We deduce that the closed ball of radius $r / (2 \sqrt{d})$ belongs to the convex hull of the points $y_{i, j}$. The growth estimate now follows by combining this fact with the estimate \cref{eq:growth_estimate_proof_upper_bound}, and using that $\psi_\mu(\,\cdot\, - v)$ is the Legendre transform of $\varphi_\mu(\,\cdot\,) + \<v, \,\cdot\,\>$.

    It remains to check that $\psi_\mu(-v) < \infty$. Note $\psi_\mu(0-v) =
    \sup_{y \in \RR^d}[\<0-v,y\>-\varphi_\mu(y)]=
    -\inf_{y \in \RR^d} [\varphi_\mu(y) + \<v, y\>]$ which is finite by \cref{eq:growth_estimate_proof_equalities}.
\end{proof}

\begin{remark}
    One is tempted to invoke~\cite[Lemma~2]{klartag2014} that gives an estimate of a convex function $\psi$ in terms of its so-called Minkowski functional $\|\,\cdot\,\|_\psi$; however, it is not clear to us whether that result can easily be used to prove \cref{prop:growth_estimate} without having a Lipschitz bound on $\psi_\mu$ which is of course a bit circular, as we are trying to establish a one-sided sub-differential estimate, in essence.
\end{remark}

\begin{lemma}
    \label{lemma:separating_hyperplane}
    Let $\mu$ be a centered probability measure on $\RR^d$, and let $e \in S^{d-1}$. There exists $w \in S^{d-1}$ and two measures $\mu_1$, $\mu_2$ on $\RR^d$ with positive mass such that $\mu = \mu_1 + \mu_2$, $\<w, y\> \geq 0$ for all $y \in \supp\,\mu_1$, $\<w, y\> \leq 0$ for all $y \in \supp\,\mu_2$, and the barycenters of $\mu_1$ and $\mu_2$ belong to the line $\spn \{e\}$.
\end{lemma}

\begin{proof}
    Let us regularize $\mu$ by convolution. Let $u \in C^\infty(\RR^d,\RR_+)$ be a radial function supported on $B_1$ with $\int_{\RR^d} u(y)\, d y = 1$. For any $\varepsilon > 0$, let
    \begin{align*}
        u^\varepsilon &\colon y \mapsto \varepsilon^{-1/d} u(y / \varepsilon), &
        f^\varepsilon &\colon y \mapsto \int_{\RR^d} u^\varepsilon(y - z)\, d \mu(z),
    \end{align*}
    and let $\mu^\varepsilon$ denote the absolutely continuous measure with density $f^\varepsilon$. It is verified that $\mu^\varepsilon$ is a centered probability measure, and that the following uniform integrability condition holds:
    \begin{equation}
        \label{eq:separating_hyperplane_proof_uniform_integrability}
        \lim_{L \to \infty} \limsup_{\varepsilon \to 0} \int_{\RR^d \setminus B_L} \|y\|\, d \mu^\varepsilon(y) = 0.
    \end{equation}

    Let $(e_i)_{1 \leq i \leq d}$ denote the canonical basis of $\RR^d$. Assume without loss of generality that
    \begin{equation*}
        e = e_d.
    \end{equation*}

    For any $v \in S^{d-1}$, let $H_v$ denote the closed half-space $\{y \in \RR^d \mid \<v, y\> \geq 0\}$. Let us consider the function
    \begin{equation*}
        F^\varepsilon \colon S^{d-1} \to \RR^{d-1}, \quad v \mapsto \left(\int_{H_v} \<e_i, y\>\, d \mu^\varepsilon(y)\right)_{1 \leq i \leq d-1}.
    \end{equation*}
    Since $\mu^\varepsilon$ is centered, the function $F^\varepsilon$ is odd (i.e., $F^\varepsilon(v) + F^\varepsilon(-v) = 0$). By the dominated convergence theorem and the absolute continuity of $\mu^\varepsilon$, the function $F^\varepsilon$ is also continuous. Then, the Borsuk--Ulam theorem implies that there exists $w^\varepsilon \in S^{d-1}$ such that $F^\varepsilon(w^\varepsilon) = 0$.

    For $\varepsilon > 0$, let $\mu_1^\varepsilon := \mu^\varepsilon \resmes H_{w^\varepsilon}$ and $\mu_2^\varepsilon := \mu^\varepsilon \resmes H_{-w^{\varepsilon}}$. Then $\mu^\varepsilon = \mu_1^\varepsilon + \mu_2^\varepsilon$, and, since $F^\varepsilon(w^\varepsilon) = 0$, the barycenters of $\mu_1^\varepsilon$ and $\mu_2^\varepsilon$ belong to $\spn \{e_d\}$.

    By the compactness of $S^{d-1}$ and Prokhorov's theorem, there exists a decreasing sequence $(\varepsilon_n)_{n \geq 0}$ approaching zero, a vector $w \in S^{d-1}$, and measures $\mu_1$, $\mu_2$ on $\RR^d$ such that $w^{\varepsilon_n}$ converges to $w$ and, for $j \in \{1, 2\}$, $\mu_j^{\varepsilon_n}$ converges weakly to $\mu_j$. Testing against appropriate bounded continuous functions, one has $\mu = \mu_1 + \mu_2$, $\<w, y\> \geq 0$ for all $y \in \supp\,\mu_1$, and $\<w, y\> \leq 0$ for all $y \in \supp\,\mu_2$.

    Using also the uniform integrability property \cref{eq:separating_hyperplane_proof_uniform_integrability}, one has, for $1 \leq i \leq d-1$ and $1 \leq j \leq 2$, $\int_{\RR^d} \<e_i, y\>\, d \mu_j(y) = \lim_{n \to \infty} \int_{\RR^d} \<e_i, y\>\, d \mu_j^{\varepsilon_n}(y) = 0$, so the barycenters of $\mu_1$ and $\mu_2$ belong to $\spn \{e_d\}$.

    We did not yet prove that $\mu_1$ and $\mu_2$ have positive mass, but since $\mu$ is centered, if, e.g., $\mu_1(\RR^d) = 0$, then one has $\<w, y\> = 0$ for all $y \in \supp\,\mu_2$, and thus one can replace $\mu_1$ and $\mu_2$ by two equal measures, i.e., $\mu_2 / 2$ and $\mu_2 / 2$. This concludes the proof.
\end{proof}

\begin{remark}
    Conceptually, assuming without loss of generality that $e = e_d$, the above proof amounts to applying a version of the ham sandwich theorem to the signed measures $\<e_i, y\>\, d \mu(y)$ for $1 \leq i \leq d-1$ and to, e.g., the uniform measure on the unit ball $B_1$. See also~\cite{cox1984}, where a similar regularization argument is used to prove a version of the ham sandwich theorem for general measures. \Cref{lemma:separating_hyperplane} also brings to mind a trick used by Lehec~\cite[section~3]{lehec2009}.
\end{remark}

\subsection{Proof of the main estimate}
\label{subsec:proof_main_estimate}

The proof of \cref{thm:stability} now results from a simple combination of \cref{prop:lipschitz_lipschitz_estimate,prop:lipschitz_nonlipschitz_estimate}.

\begin{proof}[Proof of \cref{thm:stability}]
    By a simple rescaling argument \cref{eq:invarianceMMPeq}, assume that $\mu$ and $\nu$ are probability measures. Together with \cref{eq:totalmassepsiEq} this means $\|e^{-\psi_\mu(\,\cdot\,)} - e^{-\psi_\nu(\,\cdot\, - v)}\|_{L^1(\RR^d)}$ cannot be greater than $2$; accordingly, we may assume without loss of generality that $W_1(\mu, \nu) / R$ is smaller than an arbitrary positive constant, as long as this constant depends only on $d$ and $R/r$.

    For now, let us assume that $W_1(\mu, \nu) / R \leq r/(2R)$. Then, using the Kantorovich--Rubinstein duality formula \cref{eq:kantorovich_rubinstein},
    \begin{align*}
        \int_{\RR^d} \|y\|\, d \nu(y)
        &= \int_{\RR^d} \|y\|\, d \mu(y) - \int_{\RR^d} \|y\|\, d (\mu - \nu)(y)
        \leq \int_{\RR^d} \|y\|\, d \mu(y) + W_1(\mu, \nu) \\
        &\leq R + \frac{r}{2}
        \leq \frac{3}{2} R
    \end{align*}
    and, similarly,
    \begin{equation*}
        \inf_{w \in S^{d-1}} \int_{\RR^d} |\<w, y\>|\, d \nu(y)
        \geq \inf_{w \in S^{d-1}} \int_{\RR^d} |\<w, y\>|\, d \mu(y) - W_1(\mu, \nu)
        \geq \frac{r}{2}.
    \end{equation*}

    By \cref{prop:lipschitz_lipschitz_estimate,prop:lipschitz_nonlipschitz_estimate} with $K = \overline {B_L}$ for some $L > 0$: there exists $v$, $v_2$, $v_3 \in \RR^d$ such that
    \begin{align*}
        \|e^{-\psi_\mu(\,\cdot\,)} - e^{-\psi_\nu(\,\cdot\, - v)}\|_{L^1(\RR^d)}
        &\leq \|e^{-\psi_\mu(\,\cdot\,)} - e^{-\psi_{\mu, \overline {B_L}}(\cdot - v_2)}\|_{L^1(\RR^d)} \\
        &\quad + \|e^{-\psi_\nu(\,\cdot\, - v)} - e^{-\psi_{\nu, \overline {B_L}}(\cdot - v_3)}\|_{L^1(\RR^d)} \\
        &\quad + \|e^{-\psi_{\mu, \overline {B_L}}(\cdot - v_2)} - e^{-\psi_{\nu, \overline {B_L}}(\cdot - v_3)}\|_{L^1(\RR^d)} \\
        &\lesssim_{d, R/r} \left(1 + (r L)^{(d-1)/2}\right) e^{-r L / (8 \sqrt{d})} + (L W_1(\mu, \nu))^{1/2}.
    \end{align*}
    Choosing $L = \frac{8 \sqrt{d}}{r} \log (1 + R / W_1(\mu, \nu))$, this yields
    \begin{equation*}
        \|e^{-\psi_\mu(\,\cdot\,)} - e^{-\psi_\nu(\,\cdot\, - v)}\|_{L^1(\RR^d)} \lesssim_{d, R/r} \left(\frac{W_1(\mu, \nu)}{R} \log \left(1 + \frac{R}{W_1(\mu, \nu)}\right)\right)^{1/2},
    \end{equation*}
    as announced.
\end{proof}

\section{Semidiscrete MMP}
\label{sec:semidiscrete_mmp}

In this section, we assume that $\nu$ is a centered probability measure on $\RR^d$ that is not supported on a hyperplane and that has \emph{finite support}. The measure $\nu$ should be thought of as approximating the measure $\mu$ for which we are trying to solve the MMP, i.e., find $\mm^{-1}(\mu)$ \cref{eq:mminvEq}. Then, according to \cref{thm:stability}, $\psi_\nu$ \cref{eq:psimu} is an approximation of $\psi_\mu$ (modulo a translation). The idea is then to recast the MMP for the finitely-supported $\nu$ as the minimization of an energy over $\RR^N$. The latter, in turn, is equivalent to considering piecewise affine competitors for a bona fide MMP. Here, one makes use of Legendre duality to go back and forth between vectors and such competitors. This approach gives a solution of the MMP for such $\nu$ (\cref{prop:discrete_minimizer}).

The discrete problem
and its associated energy functional $E_\nu$ is introduced in \cref{subsec:discrete_problem}, and its relation to the original (continuous) MMP and its energy $\cE_{\nu}$ is described in \cref{prop:continuous_discrete_relation,prop:discrete_minimizer}. In \cref{subsec:derivatives_discrete_energy}, we compute the first and second derivatives of the discrete energy functional $E_\nu$. This is motivated by \cref{subsec:damped_newton}, in which we describe a damped Newton method for solving the discrete MMP.

\subsection{Discrete problem}
\label{subsec:discrete_problem}

We need to introduce some notation. We denote by $N \in \NN$ the number of elements in $\supp\,\nu$, and we let $(y_i)_{1 \leq i \leq N}$ be such that $\supp\,\nu = \{y_i \mid 1 \leq i \leq N\}$.

For any vector $\Phi = (\Phi_i)_{1 \leq i \leq N} \in \RR^N$, we define $\Phi^* \colon \RR^d \to \RR$ by
\begin{equation}
    \label{eq:discrete_legendre}
    \Phi^*(x) := \max_{1 \leq i \leq N} (\<x, y_i\> - \Phi_i).
\end{equation}
The function $\Phi^*$ may be interpreted as the Legendre transform of the function assigning the values $\Phi_i$ to the points $y_i$ and the value $\infty$ to all points of $\RR^d \setminus \supp\,\nu$. Observe that $\Phi^{**}$, the Legendre transform of $\Phi^*$, is the largest convex function $\varphi \colon \RR^d \to \RR \cup \{\infty\}$ satisfying $\varphi(y_i) \leq \Phi_i$ for all $1 \leq i \leq N$.

Let us define discrete counterparts to the functionals $\cI$ and $\cE_\nu$ from \cref{eq:cI} and \cref{eq:cE_mu} by letting, for any $\Phi \in \RR^N$,
\begin{align*}
    I(\Phi) &:= -\log \int_{\RR^d} e^{-\Phi^*(x)}\, d x, &
    E_\nu(\Phi) &:= \sum_{1 \leq i \leq N} \Phi_i \nu(\{y_i\}) + I(\Phi).
\end{align*}

The functional $E_\nu$ is related to $\cE_\nu$ as follows.

\begin{proposition}
    \label{prop:continuous_discrete_relation}
    For any $\Phi \in \RR^n$, one has $\cE_\nu(\Phi^{**}) \leq E_\nu(\Phi)$. Moreover, one has $\cE_\nu(\Phi^{**}) = E_\nu(\Phi)$ if and only if $\Phi_i = \Phi^{**}(y_i)$ for all $1 \leq i \leq N$.
\end{proposition}

\begin{proof}
    The function $\Phi^*$ is convex and lower semicontinuous. Therefore, $\Phi^{***} = \Phi^*$, and thus $\cI(\Phi^{**}) = I(\Phi)$. We conclude the proof using that $\Phi^{**}(y_i) \leq \Phi_i$ for all $1 \leq i \leq N$.
\end{proof}

We can now show that, for the finitely supported measure $\nu$, the moment measure problem reduces to the problem of minimizing $E_\nu$.

\begin{proposition}
    \label{prop:discrete_minimizer}
    The function $E_\nu$ admits a minimizer on $\RR^N$, and for any such minimizer $\Phi \in \RR^N$, $\Phi^{**}$ coincides with $\varphi_\nu$ up to addition of an affine function.
\end{proposition}

\begin{proof}
    Using \cref{prop:continuous_discrete_relation} and then the fact that $\varphi_\nu$ minimizes $\cE_\nu$ (see \cref{prop:energy_minimization}), one has, for any $\Phi \in \RR^N$,
    \begin{equation}
        \label{eq:discrete_minimizer_proof_inequalities}
        E_\nu(\Phi) \geq \cE_\nu(\Phi^{**}) \geq \cE_\nu(\varphi_\nu).
    \end{equation}

    Let $\Phi_\nu := (\varphi_\nu(y_i))_{1 \leq i \leq N}$. Observe that, for all $x \in \RR^d$,
    \begin{equation*}
        (\Phi_\nu)^*(x) = \max_{1 \leq i \leq N} (\<x, y_i\> - \varphi_\nu(y_i)) \leq \sup_{y \in \RR^d} (\<x, y\> - \varphi_\nu(y)) = (\varphi_\nu)^*(x).
    \end{equation*}
    Thus, $I(\Phi_\nu) \leq \cI(\varphi_\nu)$, and $E_\nu(\Phi_\nu) \leq \cE_\nu(\varphi_\nu)$. Using \cref{eq:discrete_minimizer_proof_inequalities}, $\Phi_\nu$ minimizes $E_\nu$, and $E_\nu(\Phi_\nu) = \cE_\nu(\varphi_\nu)$.

    For any minimizer $\Phi \in \RR^N$ of $E_\nu$, one has $E_\nu(\Phi) = E_\nu(\Phi_\nu) = \cE_\nu(\varphi_\nu)$. Then, the inequalities in \cref{eq:discrete_minimizer_proof_inequalities} become equalities. This implies that $\Phi^{**}$ minimizes $\cE_\nu$. By \cref{prop:energy_minimization}, $\Phi^{**}$ has moment measure $\lambda \nu$, where $\lambda = \int_{\RR^d} e^{-\Phi^*(x)}\, d x$. Using \cref{thm:existence_uniqueness}, we deduce that $\Phi^{**}$ coincides with $\varphi_\nu$ up to addition of an affine function, as announced.
\end{proof}

\subsection{Derivatives of the discrete energy and Laguerre diagrams}
\label{subsec:derivatives_discrete_energy}

Given $\Phi \in \RR^N$, all information about the convex functions $\Phi^*$ and $\Phi^{**}$ can be computed from the vector $\Phi$, for instance using a standard convex hull algorithm and the fact the epigraph of $\Phi^{**}$ is the unbounded convex polygon generated by the points $\{(y_i, \Phi_i) \mid 1 \leq i \leq N\}$ and by the ray spanned by the last element of the canonical basis of $\RR^{d+1}$.

For any $\Phi \in \RR^N$ and $1 \leq i \leq N$, let us define the \emph{Laguerre cell}
\begin{equation}
    \label{eq:laguerre_cell}
    \Lag_i(\Phi) := \{x \in \RR^d \mid \forall 1 \leq j \leq N,\, \<x, y_j - y_i\> \leq \Phi_j - \Phi_i\}.
\end{equation}
Observe that $\Lag_i(\Phi_i)$ is either empty or a (potentially unbounded) polyhedron of dimension $d$ or less. It follows from the definition \cref{eq:discrete_legendre} of $\Phi^*$ that
\begin{equation}
    \label{eq:legendre_on_cell}
    \Phi^* = \<y_i, \cdot\> - \Phi_i \quad \text{on } \Lag_i(\Phi).
\end{equation}
Moreover, one has $\bigcup_{1 \leq i \leq N} \Lag_i(\Phi) = \RR^d$, and $\inter(\Lag_i(\Phi)) \cap \inter(\Lag_j(\Phi)) = \emptyset$ whenever $i \neq j$. The resulting polyhedral partition of $\RR^d$ is typically called the \emph{Laguerre diagram}, or \emph{power diagram}, associated to the points $y_i$ and the weights $\Phi_i$. We refer to~\cite{kitagawa2019} for a more detailed overview of these notions and how they are used in the setting of semidiscrete optimal transport. Let us only mention here that the notion of Laguerre cells is related to the subdifferential of $\Phi^{**}$: if $x \in \Lag_i(\Phi)$, then $y_i \in \partial \Phi^*(x)$, and thus $x \in \partial \Phi^{**}(y_i)$.

Let us study the first and second derivatives of the energy $E_\nu$. A key observation used in the analysis below is that adjusting the values of the elements of the vector $\Phi$ amounts to adjusting the $N$ hyperplanes involved in the definition \cref{eq:discrete_legendre} of $\Phi^*$.

\begin{lemma}
    \label{lemma:integral_derivative}
    For any $\Phi \in \RR^N$ and $1 \leq i \leq N$, one has
    \begin{equation}
        \label{eq:integral_derivative}
        \frac{\partial}{\partial \Phi_i} \int_{\RR^d} e^{-\Phi^*(x)}\, d x = \int_{\Lag_i(\Phi)} e^{-\Phi^*(x)}\, d x.
    \end{equation}
\end{lemma}

\begin{proof}
    Let $(e_j)_{1 \leq j \leq N}$ denote the canonical basis of $\RR^N$. It follows from \cref{eq:laguerre_cell} that, for any $1 \leq j \leq N$ and $x \in \inter(\Lag_j(\Phi))$, there exists $c_{\Phi, x} > 0$, depending only on $\Phi$ and $x$, such that, for all $h \in [-c_{\Phi, x}, c_{\Phi, x}]$, one has $x \in \Lag_j(\Phi + h e_i)$. Using also \cref{eq:legendre_on_cell}, we deduce that
    \begin{equation*}
        \frac{\partial}{\partial \Phi_i} e^{-\Phi^*(\,\cdot\,)}
        = \frac{\partial}{\partial \Phi_i} e^{\Phi_i - \<\cdot, y_i\>}
        = e^{\Phi_i - \<\cdot, y_i\>}
        = e^{-\Phi^*(\,\cdot\,)}
        \quad \text{on } \inter(\Lag_i(\Phi))
    \end{equation*}
    and, for $j \neq i$,
    \begin{equation*}
        \frac{\partial}{\partial \Phi_i} e^{-\Phi^*(\,\cdot\,)}
        = \frac{\partial}{\partial \Phi_i} e^{\Phi_j - \<y_i, \cdot\>}
        = 0
        \quad \text{on } \inter(\Lag_j(\Phi)).
    \end{equation*}
    On the other hand $e^{-\Phi^*(x)}$ is not necessarily differentiable with respect to $\Phi$ if $x$ belongs to the Lebesgue negligible set $\RR^d \setminus \bigcup_{1 \leq j \leq N} \inter(\Lag_j(\Phi))$. Still, it follows from the definition \cref{eq:discrete_legendre} of $\Phi^*$ that, for all $h \in \RR \setminus \{0\}$ and $x \in \RR^d$, one has $\Phi^*(x) - |h| \leq (\Phi + h e_i)^*(x) \leq \Phi^*(x) + |h|$, and thus
    \begin{equation*}
        \left|\frac{e^{-(\Phi + h e_i)^*(x)} - e^{-\Phi^*(x)}}{h}\right|
        \leq \frac{\max \{e^{|h|} - 1, 1 - e^{|h|}\}}{|h|} e^{-\Phi^*(x)}
        = \frac{e^{|h|} - 1}{|h|} e^{-\Phi^*(x)}.
    \end{equation*}
    Then, by the dominated convergence theorem,
    \begin{align*}
        \frac{\partial}{\partial \Phi_i} \int_{\RR^d} e^{-\Phi^*(x)}\, d x
        &= \lim_{\substack{h \to 0 \\ h \neq 0}} \int_{\RR^d} \frac{e^{-(\Phi + h e_i)^*(x)} - e^{-\Phi^*(x)}}{h}\, d x \\
        &= \sum_{1 \leq j \leq N} \int_{\inter(\Lag_j(\Phi))} \frac{\partial}{\partial \Phi_i} e^{-\Phi^*(x)}\, d x
        = \int_{\Lag_i(\Phi)} e^{-\Phi^*(x)}\, d x,
    \end{align*}
    as announced.
\end{proof}

\begin{proposition}
    For any $\Phi \in \RR^N$ and $1 \leq i \leq N$, one has
    \begin{equation}
        \label{eq:first_derivative}
        \frac{\partial E_\nu}{\partial \Phi_i}(\Phi) = \nu(\{y_i\}) - \frac{\int_{\Lag_i(\Phi)} e^{-\Phi^*(x)}\, d x}{\int_{\RR^d} e^{-\Phi^*(x)}\, d x}.
    \end{equation}
\end{proposition}

\begin{proof}
    This follows directly from \cref{lemma:integral_derivative} and the definition of $E_\nu$.
\end{proof}

Let us introduce the set $U \subset \RR^N$ defined as
\begin{equation*}
    U := \{\Phi \in \RR^N \mid \forall 1 \leq i \leq N,\, \inter(\Lag_i(\Phi)) \neq \emptyset\}.
\end{equation*}

\begin{proposition}
    \label{prop:minimizer_in_U}
    If $\Phi \in \RR^N$ minimizes $E_\nu$, then $\Phi \in U$.
\end{proposition}

\begin{proof}
    This follows from the first-order optimality condition $\nabla E_\nu(\Phi) = 0$, where $\nabla E_\nu(\Phi)$ is given by the formula \cref{eq:first_derivative}, and from the fact that $\nu(\{y_i\}) > 0$ for all $1 \leq i \leq N$.
\end{proof}

\begin{proposition}
    Let $\Phi \in U$. For any $1 \leq i \leq N$, one has
    \begin{align}
        \label{eq:second_derivative_diagonal}
        \frac{\partial^2 E_\nu}{\partial \Phi_i^2}(\Phi)
        &= \left(\frac{\int_{\Lag_i(\Phi)} e^{-\Phi^*(x)}\, d x}{\int_{\RR^d} e^{-\Phi^*(x)}\, d x}\right)^2 - \frac{\int_{\Lag_i(\Phi)} e^{-\Phi^*(x)}\, d x}{\int_{\RR^d} e^{-\Phi^*(x)}\, d x} \\
        \notag
        &\quad + \sum_{j \neq i} \frac{\int_{\Lag_i(\Phi) \cap \Lag_j(\Phi)} e^{-\Phi^*(x)}\, d \cH^{d-1}(x)}{\|y_i - y_j\| \int_{\RR^d} e^{-\Phi^*(x)}\, d x}.
    \end{align}
    For any $1 \leq i$, $j \leq N$ with $i \neq j$, one has
    \begin{align}
        \label{eq:second_derivative_extradiagonal}
        \frac{\partial^2 E_\nu}{\partial \Phi_i \partial \Phi_j}(\Phi)
        &= \frac{\left(\int_{\Lag_i(\Phi)} e^{-\Phi^*(x)}\, d x\right) \left(\int_{\Lag_j(\Phi)} e^{-\Phi^*(x)}\, d x\right)}{\left(\int_{\RR^d} e^{-\Phi^*(x)}\, d x\right)^2} \\
        \notag
        &\quad - \frac{\int_{\Lag_i(\Phi) \cap \Lag_j(\Phi)} e^{-\Phi^*(x)}\, d \cH^{d-1}(x)}{\|y_i - y_j\| \int_{\RR^d} e^{-\Phi^*(x)}\, d x}.
    \end{align}
\end{proposition}

\begin{proof}
    Remember that, using \cref{eq:legendre_on_cell}, one has, for any $1 \leq i \leq N$,
    \begin{equation*}
        \int_{\Lag_i(\Phi)} e^{-\Phi^*(x)}\, d x = \int_{\RR^d} e^{\Phi_i - \<x, y_i\>}\, d x.
    \end{equation*}
    Studying the geometric structure of the Laguerre cells $\Lag_j(\Phi)$, for $1 \leq j \leq N$, as described by equation \cref{eq:laguerre_cell}, and using the fact that $\Phi \in U$, one obtains
    \begin{equation*}
        \frac{\partial}{\partial \Phi_i} \int_{\Lag_i(\Phi)} e^{-\Phi^*(x)}\, d x = \int_{\Lag_i(\Phi)} e^{-\Phi^*(x)}\, d x - \sum_{j \neq i} \frac{\int_{\Lag_i(\Phi) \cap \Lag_j(\Phi)} e^{-\Phi^*(x)}\, d \cH^{d-1}(x)}{\|y_i - y_j\|}
    \end{equation*}
    and, for $j \neq i$,
    \begin{equation*}
        \frac{\partial}{\partial \Phi_j} \int_{\Lag_i(\Phi)} e^{-\Phi^*(x)}\, d x = \frac{\int_{\Lag_i(\Phi) \cap \Lag_j(\Phi)} e^{-\Phi^*(x)}\, d \cH^{d-1}(x)}{\|y_i - y_j\|}.
    \end{equation*}
    We easily conclude, using the above results and equations \cref{eq:integral_derivative} and \cref{eq:first_derivative}.
\end{proof}

\subsection{Damped Newton method}
\label{subsec:damped_newton}

In view of the above computations, one can attempt to compute a minimizer of $E_\nu$ using a damped Newton method, that we describe as \cref{algo:damped_newton}. Observe that this method is largely inspired by a similar approach used in the setting of the optimal transform problem, see e.g.~\cite{kitagawa2019}.

\begin{algorithm}
    \caption{Damped Newton method for computing $\psi_\nu$}
    \label{algo:damped_newton}
    \begin{algorithmic}
        \Require a tolerance $\varepsilon > 0$, an initialization $\Phi^{(0)} \in U$.
        \For{$k \in \NN$}
            \If{$\|\nabla E_\nu(\Phi^{(k)})\| \leq \varepsilon \|(\nu(\{y_i\}))_{1 \leq i \leq N}\|$}
                \State \textbf{return} $(\Phi^{(k)})^* + \log \int_{\RR^d} e^{-(\Phi^{(k)})^*(x)}\, d x$
            \EndIf
            \State \textbf{find} $d^{(k)} \in \RR^N$ \textbf{s.t.}\ $\nabla^2 E_\nu(\Phi^{(k)}) d^{(k)} = -\nabla E_\nu(\Phi^{(k)})$
            \State $\tau^{(k)} \gets \max \{2^{-i} \mid i \in \NN,\, \Phi^{(k)} + 2^{-i} d^{(k)} \in U\}$
            \State $\Phi^{(k+1)} \gets \Phi^{(k)} + \tau^{(k)} d^{(k)}$
        \EndFor
    \end{algorithmic}
\end{algorithm}

\Cref{algo:damped_newton} is a damped Newton method, since the step size is sometimes reduced to ensure that all iterates belong to the set $U$. This is motivated by the fact that the formulas \cref{eq:second_derivative_diagonal} and \cref{eq:second_derivative_extradiagonal} for the elements of $\nabla^2 E_\nu(\Phi)$ are only valid when $\Phi \in U$. Remember that \cref{prop:minimizer_in_U} guarantees that the minimizer of $E_\nu$ belongs to $U$.

The matrix $\nabla^2 E_\nu(\Phi^{(k)})$ appearing in \cref{algo:damped_newton} is not invertible, since $E_\nu$ is invariant under addition of a vector of the form $(a + \<v, y_i\>)_{1 \leq i \leq N}$, for $a \in \RR$ and $v \in \RR^d$. In practice, we rather solve the linear system $M_\nu(\Phi^{(k)}) d^{(k)} = -\nabla E_\nu(\Phi^{(k)})$ with $M_\nu(\Phi^{(k)}) := \nabla^2 E_\nu(\Phi^{(k)}) + u u^\top + \sum_{i=1}^d v_i v_i^\top$, where $u$ denotes the vector of size $N$ whose all elements are one, and $v_i$ denotes the vector $((y_j)_i)_{1 \leq j \leq N}$. According to \cref{eq:second_derivative_diagonal} and \cref{eq:second_derivative_extradiagonal}, $\nabla^2 E_\nu(\Phi^{(k)})$ is the sum of a sparse matrix and a matrix of rank one, and thus $M_\nu(\Phi^{(k)})$ is the sum of a sparse matrix and a matrix of rank $d+2$. Elements of $\nabla E_\nu(\Phi^{(k)})$ and $\nabla^2 E_\nu(\Phi^{(k)})$ (and thus $M_\nu(\Phi^{(k)})$) can be computed using \cref{eq:first_derivative}, \cref{eq:second_derivative_diagonal}, \cref{eq:second_derivative_extradiagonal}, and exact integration formulas for the integrals over Laguerre cells.

The question of the invertibility of the matrix $M_\nu(\Phi^{(k)})$, and of the convergence of the method described in \cref{algo:damped_newton}, is left for future work. The analysis of similar Newton methods in the setting of optimal transport~\cite{kitagawa2019} suggests that it may be useful to modify the damping criterion to impose that the iterates are sufficiently far from the boundary of $U$; such considerations are also left for future work. In practice, on our experiments, we do not observe any issues with the damping criterion, nor with solving the linear system $M_\nu(\Phi^{(k)}) d^{(k)} = -\nabla E_\nu(\Phi^{(k)})$ using a sparse linear solver.

\section{Numerical experiments}
\label{sec:numerical_experiments}

We consider different measures $\mu$ for which an exact formula for an essentially continuous convex function $\psi_\mu \colon \RR^d \to \RR \cup \{\infty\}$ with moment measure $\mu$, and for its Legendre transform $\varphi_\mu$, is known. In each case, the translation up to which $\psi_\mu$ is unique will be prescribed by the formulas we give for $\psi_\mu$ and $\varphi_\mu$. We also consider different rules for choosing a finitely supported approximation $\nu$ of the measure $\mu$.

Once $\nu$ is chosen, we compute $\psi_\nu$ using \cref{algo:damped_newton}, with tolerance $\varepsilon = 10^{-10}$. \Cref{algo:damped_newton} only determines $\psi_\nu$ up to a translation, which amounts to adding a linear function to $\varphi_\nu$. As a post-processing step, we choose this translation and linear function so as to minimize the $L^2$ norm $\|\varphi_\mu - \varphi_\nu\|_{L^2(\nu)}$.

\paragraph{Test case~1}

We choose $\mu$ as the uniform probability measure on the square $[-1, 1]^2$. In this setting, valid formulas for $\psi_\mu$ and $\varphi_\mu$ are
\begin{align*}
    \psi_\mu(x) &= 2 \log(1 + e^{x_1}) - x_1 + 2 \log(1 + e^{x_2}) - x_2, \\
    \varphi_\mu(y) &= (1 + y_1) \log (1 + y_1) + (1 - y_1) \log (1 - y_1) \\
    &\quad + (1 + y_2) \log (1 + y_2) + (1 - y_2) \log (1 - y_2) - 4 \log 2,
\end{align*}
using the convention that $t \log t = 0$ when $t = 0$ and $t \log t = \infty$ when $t < 0$. Given a parameter $n \in 2 \NN^*$ describing the fineness of the discretization of $\mu$, and letting $h := 2/n$, we choose $\nu$ as
\begin{equation*}
    \nu := \frac{1}{(n + 1)^2} \sum_{j_1=-\frac{n}{2}}^{\frac{n}{2}} \sum_{j_2=-\frac{n}{2}}^{\frac{n}{2}} \delta_{(j_1 h, j_2 h)}.
\end{equation*}
Observe that in this setting, the number $N$ of points of $\supp\,\nu$ is $(n+1)^2$, and the Wasserstein distance $W_1(\mu, \nu)$ scales like $h$, i.e., like $n^{-1}$ and like $N^{-1/2}$.

\paragraph{Test case~2}

We choose $\mu$ as the uniform probability measure on the triangle with vertices $(-1, -1)$, $(2, -1)$, and $(-1, 2)$. In this setting, valid formulas for $\psi_\mu$ and $\varphi_\mu$ are
\begin{align*}
    \psi_\mu(x) &= 3 \log (1 + e^{x_1} + e^{x_2}) - x_1 - x_2 - \log 2, \\
    \varphi_\mu(y) &= (1 + y_1) \log (1 + y_1) + (1 + y_2) \log (1 + y_2) \\
    &\quad + (1 - y_1 - y_2) \log (1 - y_1 - y_2) - 3 \log 3 + \log 2,
\end{align*}
using the convention that $t \log t = 0$ when $t = 0$ and $t \log t = \infty$ when $t < 0$. Given a discretization parameter $n \in 2 \NN^*$, and letting $h := 2/n$, we choose $\nu$ as
\begin{equation*}
    \nu := \frac{8}{(3n+2)(3n+4)} \sum_{j_1=-\frac{n}{2}}^{n} \sum_{j_2=-\frac{n}{2}}^{\frac{n}{2}-j_1} \delta_{(j_1 h, j_2 h)}.
\end{equation*}
Then, the number $N$ of points of $\supp\,\nu$ is $(3n+2)(3n+4)/8$, and the Wasserstein distance $W_1(\mu, \nu)$ scales like $h$, i.e., like $n^{-1}$ and like $N^{-1/2}$.

\paragraph{Test case~3} We consider again the measure $\mu$ and the functions $\psi_\mu$ and $\varphi_\mu$ from test case~1, but experiment with a different rule for choosing the finitely supported measure $\nu$. As in test case~1, we let $n \in 2 \NN^*$ be a discretization parameter, and let $h := 2/n$. For any pair $(j_1, j_2) \in \ZZ^2$, let us define the basis function
\begin{align}
    \label{eq:basis_function}
    u_{j_1, j_2}^h &\colon \RR^d \to \RR,\, y \mapsto \max \bigg\{0, 1 - \max \bigg\{\frac{|y_1 - j_1 h|}{h}, \frac{|y_2 - j_2 h|}{h}, \\
    \notag
    &\qquad \qquad \qquad \qquad \qquad \qquad \qquad \qquad \frac{|y_1 + y_2 - j_1 h - j_2 h|}{h}\bigg\}\bigg\}.
\end{align}
Observe that the functions $u_{j_1, j_2}^h$ are continuous and piecewise affine, that
\begin{equation*}
    u_{j_1, j_2}^h(k_1 h, k_2 h) = \begin{cases}
        1 &\text{if } k_1 = j_1 \text{ and } k_2 = j_2, \\
        0 &\text{else}
    \end{cases}
\end{equation*}
for all $(k_1, k_2) \in \ZZ^2 \setminus \{(j_1, j_2)\}$, and that, for all $y \in \RR^d$, $\sum_{(j_1, j_2) \in \ZZ^2} u_{j_1, j_2}^h(y) = 1$. We choose $\nu$ as
\begin{equation}
    \label{eq:discrete_measure_test_case_3}
    \nu := \sum_{j_1=-\frac{n}{2}}^{\frac{n}{2}} \sum_{j_2=-\frac{n}{2}}^{\frac{n}{2}} \left(\int_{\RR^d} u_{j_1, j_2}^h(y)\, d \mu(y)\right) \delta_{(j_1 h, j_2 h)}.
\end{equation}
Equivalently, $\supp\,\nu$ is unchanged from test case~1, but now, given $y \in \supp\,\nu$, one has $\nu(\{y\}) = h^2 / 4$ if $y \in (-1, 1)^2$, $\nu(\{y\}) = h^2 / 8$ if $y$ lies in the relative interior of an edge of the square $[-1, 1]^2$, $\nu(\{y\}) = h^2 / 24$ if $y \in \{(-1, -1), (1, 1)\}$, and $\nu(\{y\}) = h^2 / 12$ if $y \in \{(-1, 1), (1, -1)\}$. As in test case~1, the number $N$ of points of $\supp\,\nu$ is $(n+1)^2$, and the Wasserstein distance $W_1(\mu, \nu)$ scales like $h$, i.e., like $n^{-1}$ and like $N^{-1/2}$.

The new formula \cref{eq:discrete_measure_test_case_3} for choosing $\nu$ can be motivated as follows. Using the notation of \cref{sec:semidiscrete_mmp}, for any $1 \leq i \leq N$, there exists a pair $(j_1, j_2) \in \{-n/2, \ldots, n/2\}^2$ for which $y_i = (j_1 h, j_2 h)$; let then $u_i := u_{j_1, j_2}^h$. When $\nu$ is chosen according to \cref{eq:discrete_measure_test_case_3}, it is easily verified that, for any $\Phi \in \RR^N$, one has $E_\nu(\Phi) = \cE_\mu(\sum_{i=1}^N \Phi_i u_i + \cvxind {[-1, 1]^2})$. Thus, the numerical solution $\psi_\nu$ is the Legendre transform of a minimizer of $\cE_\mu$ over the finite dimensional vector space of functions $\{\sum_{i=1}^N \Phi_i u_i + \cvxind {[-1, 1]^2} \mid \Phi \in \RR^N\}$, which can be interpreted as a $P1$ finite element space. Accordingly, with the specific choice of $\nu$ given by \cref{eq:discrete_measure_test_case_3}, the numerical method that we use to approximate $\varphi_\mu$ and $\psi_\mu$ can be understood as a $P1$ finite element method applied to the minimization of $\cE_\mu$.

\paragraph{Test case~4} We adapt the approach of test case~3 to the problem of test case~2. We consider the measure $\mu$ and the functions $\psi_\mu$ and $\varphi_\mu$ from test case~2, and we choose $\nu$ as
\begin{equation*}
    \nu := \sum_{j_1=-\frac{n}{2}}^{n} \sum_{j_2=-\frac{n}{2}}^{\frac{n}{2}-i} \left(\int_{\RR^d} u_{j_1, j_2}^h(y)\, d \mu(y)\right) \delta_{(j_1 h, j_2 h)},
\end{equation*}
where the basis functions $u_{j_1, j_2}^h$ are defined as in \cref{eq:basis_function}. Equivalently, $\supp\,\nu$ is unchanged from test case~2, but now, given $y \in \supp\,\nu$, one has $\nu(\{y\}) = 2 h^2 / 9$ if $y$ lies in the relative interior of the triangle $T := \Conv(\{(-1, -1), (2, -1), (-1, 2)\})$, $\nu(\{y\}) = h^2 / 9$ if $y$ lies in the relative interior of an edge of $T$, and $\nu(\{y\}) = h^2 / 27$ if $y$ is a vertex of $T$. As in test case~2, the number $N$ of points of $\supp\,\nu$ is $(3n+2) (3n+4) / 8$, and the Wasserstein distance $W_1(\mu, \nu)$ scales like $h$, i.e., like $n^{-1}$ and like $N^{-1/2}$.

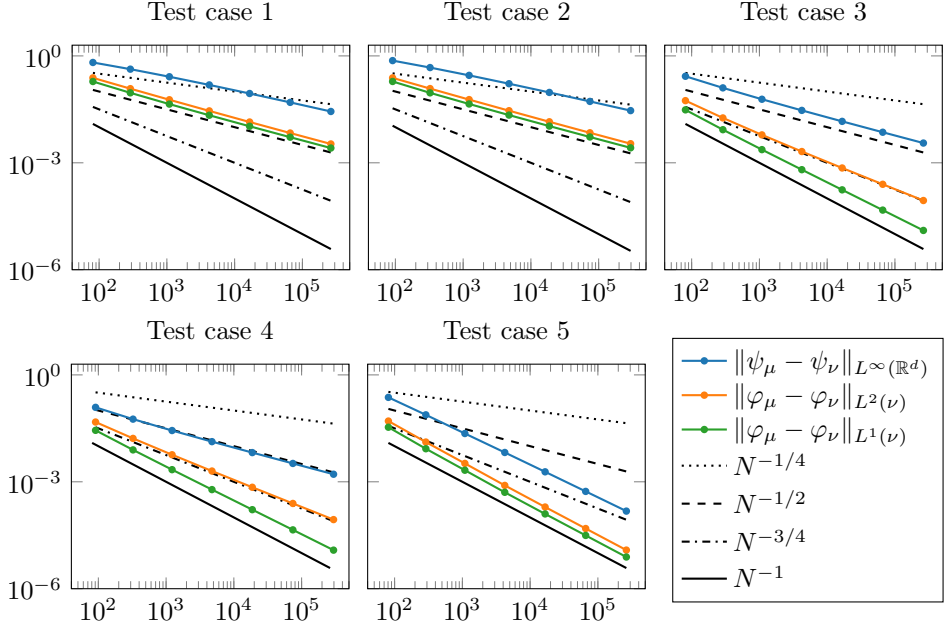
\begin{figure}[t]
    \centering
    \begin{tikzpicture}[baseline=(current bounding box.center)]
        \begin{loglogaxis}[width=150pt,height=130pt,mark size=1pt,xmin=40,xmax=500000,ymin=1e-6,ymax=2,title=Test case~1]
            \addplot[thick,dotted] table[x=N,y expr={pow(\thisrow{N},-1/4)}] {results/test1.txt};
            \addplot[thick,dashed] table[x=N,y expr={pow(\thisrow{N},-1/2)}] {results/test1.txt};
            \addplot[thick,dashdotted] table[x=N,y expr={pow(\thisrow{N},-3/4)}] {results/test1.txt};
            \addplot[thick] table[x=N,y expr={pow(\thisrow{N},-1)}] {results/test1.txt};
            \addplot[thick,tab1,mark=*,style=thick] table[x=N,y=Linfty] {results/test1.txt};
            \addplot[thick,tab2,mark=*] table[x=N,y=L2] {results/test1.txt};
            \addplot[thick,tab3,mark=*] table[x=N,y=L1] {results/test1.txt};
        \end{loglogaxis}
    \end{tikzpicture}%
    \begin{tikzpicture}[baseline=(current bounding box.center)]
        \begin{loglogaxis}[width=150pt,height=130pt,mark size=1pt,xmin=40,xmax=500000,ymin=1e-6,ymax=2,yticklabel=\empty,title=Test case~2]
            \addplot[thick,dotted] table[x=N,y expr={pow(\thisrow{N},-1/4)}] {results/test2.txt};
            \addplot[thick,dashed] table[x=N,y expr={pow(\thisrow{N},-1/2)}] {results/test2.txt};
            \addplot[thick,dashdotted] table[x=N,y expr={pow(\thisrow{N},-3/4)}] {results/test2.txt};
            \addplot[thick] table[x=N,y expr={pow(\thisrow{N},-1)}] {results/test2.txt};
            \addplot[thick,tab1,mark=*] table[x=N,y=Linfty] {results/test2.txt};
            \addplot[thick,tab2,mark=*] table[x=N,y=L2] {results/test2.txt};
            \addplot[thick,tab3,mark=*] table[x=N,y=L1] {results/test2.txt};
        \end{loglogaxis}
    \end{tikzpicture}%
    \begin{tikzpicture}[baseline=(current bounding box.center)]
        \begin{loglogaxis}[width=150pt,height=130pt,mark size=1pt,xmin=40,xmax=500000,ymin=1e-6,ymax=2,yticklabel=\empty,title=Test case~3]
            \addplot[thick,dotted] table[x=N,y expr={pow(\thisrow{N},-1/4)}] {results/test3.txt};
            \addplot[thick,dashed] table[x=N,y expr={pow(\thisrow{N},-1/2)}] {results/test3.txt};
            \addplot[thick,dashdotted] table[x=N,y expr={pow(\thisrow{N},-3/4)}] {results/test3.txt};
            \addplot[thick] table[x=N,y expr={pow(\thisrow{N},-1)}] {results/test3.txt};
            \addplot[thick,tab1,mark=*] table[x=N,y=Linfty] {results/test3.txt};
            \addplot[thick,tab2,mark=*] table[x=N,y=L2] {results/test3.txt};
            \addplot[thick,tab3,mark=*] table[x=N,y=L1] {results/test3.txt};
        \end{loglogaxis}
    \end{tikzpicture}%

    \begin{tikzpicture}[baseline=(current bounding box.center)]
        \begin{loglogaxis}[width=150pt,height=130pt,mark size=1pt,xmin=40,xmax=500000,ymin=1e-6,ymax=2,title=Test case~4]
            \addplot[thick,dotted] table[x=N,y expr={pow(\thisrow{N},-1/4)}] {results/test4.txt};
            \addplot[thick,dashed] table[x=N,y expr={pow(\thisrow{N},-1/2)}] {results/test4.txt};
            \addplot[thick,dashdotted] table[x=N,y expr={pow(\thisrow{N},-3/4)}] {results/test4.txt};
            \addplot[thick] table[x=N,y expr={pow(\thisrow{N},-1)}] {results/test5.txt};
            \addplot[thick,tab1,mark=*] table[x=N,y=Linfty] {results/test4.txt};
            \addplot[thick,tab2,mark=*] table[x=N,y=L2] {results/test4.txt};
            \addplot[thick,tab3,mark=*] table[x=N,y=L1] {results/test4.txt};
        \end{loglogaxis}
    \end{tikzpicture}%
    \begin{tikzpicture}[baseline=(current bounding box.center)]
        \begin{loglogaxis}[width=150pt,height=130pt,mark size=1pt,xmin=40,xmax=500000,ymin=1e-6,ymax=2,yticklabel=\empty,title=Test case~5]
            \addplot[thick,dotted] table[x=N,y expr={pow(\thisrow{N},-1/4)}] {results/test5.txt};
            \addplot[thick,dashed] table[x=N,y expr={pow(\thisrow{N},-1/2)}] {results/test5.txt};
            \addplot[thick,dashdotted] table[x=N,y expr={pow(\thisrow{N},-3/4)}] {results/test5.txt};
            \addplot[thick] table[x=N,y expr={pow(\thisrow{N},-1)}] {results/test5.txt};
            \addplot[thick,tab1,mark=*] table[x=N,y=Linfty] {results/test5.txt};
            \addplot[thick,tab2,mark=*] table[x=N,y=L2] {results/test5.txt};
            \addplot[thick,tab3,mark=*] table[x=N,y=L1] {results/test5.txt};
        \end{loglogaxis}
    \end{tikzpicture}%
    \hspace{10pt}%
    \begin{tikzpicture}[baseline=(current bounding box.center)]
        \begin{axis}[mark size=1pt,xmin=0,xmax=1,ymin=0,ymax=1,hide axis,legend cell align=left]
            \addlegendimage{thick,tab1,mark=*}
            \addlegendentry{$\|\psi_\mu - \psi_\nu\|_{L^\infty(\RR^d)}$}
            \addlegendimage{thick,tab2,mark=*}
            \addlegendentry{$\|\varphi_\mu - \varphi_\nu\|_{L^2(\nu)}$}
            \addlegendimage{thick,tab3,mark=*}
            \addlegendentry{$\|\varphi_\mu - \varphi_\nu\|_{L^1(\nu)}$}
            \addlegendimage{thick,dotted}
            \addlegendentry{$N^{-1/4}$}
            \addlegendimage{thick,dashed}
            \addlegendentry{$N^{-1/2}$}
            \addlegendimage{thick,dashdotted}
            \addlegendentry{$N^{-3/4}$}
            \addlegendimage{thick}
            \addlegendentry{$N^{-1}$}
        \end{axis}
    \end{tikzpicture}%
    \caption{Approximation error, with respect to the number $N$ of points in $\supp\,\nu$ (recall that, by \cref{eq:legendre_isometry}, one has $\|\psi_\mu - \psi_\nu\|_{L^\infty(\RR^d)} = \|\varphi_\mu - \varphi_\nu\|_{L^\infty(\supp\,\mu)} \geq \|\varphi_\mu - \varphi_\nu\|_{L^\infty(\nu)}$ in all test cases).}
    \label{fig:error}
\end{figure}

\begin{figure}[t]
    \centering
    \begin{tikzpicture}[baseline=(current bounding box.center)]
        \begin{semilogyaxis}[width=150pt,height=130pt,mark size=1pt,xmin=-1,xmax=18,ymin=1e-15,ymax=100,title=Test case~1]
            \addplot[thick,dotted,domain=-100:100,update limits=false] {1e-10};
            \addplot[thick,tab1,scatter,point meta=explicit symbolic,scatter/classes={0={mark=-},1={mark=*}}] table[x=k,y=residual,meta=damping] {results/test1-n8.txt};
            \addplot[thick,tab2,scatter,point meta=explicit symbolic,scatter/classes={0={mark=-},1={mark=*}}] table[x=k,y=residual,meta=damping] {results/test1-n16.txt};
            \addplot[thick,tab3,scatter,point meta=explicit symbolic,scatter/classes={0={mark=-},1={mark=*}}] table[x=k,y=residual,meta=damping] {results/test1-n32.txt};
            \addplot[thick,tab4,scatter,point meta=explicit symbolic,scatter/classes={0={mark=-},1={mark=*}}] table[x=k,y=residual,meta=damping] {results/test1-n64.txt};
            \addplot[thick,tab5,scatter,point meta=explicit symbolic,scatter/classes={0={mark=-},1={mark=*}}] table[x=k,y=residual,meta=damping] {results/test1-n128.txt};
            \addplot[thick,tab6,scatter,point meta=explicit symbolic,scatter/classes={0={mark=-},1={mark=*}}] table[x=k,y=residual,meta=damping] {results/test1-n256.txt};
            \addplot[thick,tab7,scatter,point meta=explicit symbolic,scatter/classes={0={mark=-},1={mark=*}}] table[x=k,y=residual,meta=damping] {results/test1-n512.txt};
        \end{semilogyaxis}
    \end{tikzpicture}%
    \begin{tikzpicture}[baseline=(current bounding box.center)]
        \begin{semilogyaxis}[width=150pt,height=130pt,mark size=1pt,xmin=-1,xmax=18,ymin=1e-15,ymax=100,yticklabel=\empty,title=Test case~2]
            \addplot[thick,dotted,domain=-100:100,update limits=false] {1e-10};
            \addplot[thick,tab1,scatter,point meta=explicit symbolic,scatter/classes={0={mark=-},1={mark=*}}] table[x=k,y=residual,meta=damping] {results/test2-n8.txt};
            \addplot[thick,tab2,scatter,point meta=explicit symbolic,scatter/classes={0={mark=-},1={mark=*}}] table[x=k,y=residual,meta=damping] {results/test2-n16.txt};
            \addplot[thick,tab3,scatter,point meta=explicit symbolic,scatter/classes={0={mark=-},1={mark=*}}] table[x=k,y=residual,meta=damping] {results/test2-n32.txt};
            \addplot[thick,tab4,scatter,point meta=explicit symbolic,scatter/classes={0={mark=-},1={mark=*}}] table[x=k,y=residual,meta=damping] {results/test2-n64.txt};
            \addplot[thick,tab5,scatter,point meta=explicit symbolic,scatter/classes={0={mark=-},1={mark=*}}] table[x=k,y=residual,meta=damping] {results/test2-n128.txt};
            \addplot[thick,tab6,scatter,point meta=explicit symbolic,scatter/classes={0={mark=-},1={mark=*}}] table[x=k,y=residual,meta=damping] {results/test2-n256.txt};
            \addplot[thick,tab7,scatter,point meta=explicit symbolic,scatter/classes={0={mark=-},1={mark=*}}] table[x=k,y=residual,meta=damping] {results/test2-n512.txt};
        \end{semilogyaxis}
    \end{tikzpicture}%
    \begin{tikzpicture}[baseline=(current bounding box.center)]
        \begin{semilogyaxis}[width=150pt,height=130pt,mark size=1pt,xmin=-1,xmax=18,ymin=1e-15,ymax=100,yticklabel=\empty,title=Test case~3]
            \addplot[thick,dotted,domain=-100:100,update limits=false] {1e-10};
            \addplot[thick,tab1,scatter,point meta=explicit symbolic,scatter/classes={0={mark=-},1={mark=*}}] table[x=k,y=residual,meta=damping] {results/test3-n8.txt};
            \addplot[thick,tab2,scatter,point meta=explicit symbolic,scatter/classes={0={mark=-},1={mark=*}}] table[x=k,y=residual,meta=damping] {results/test3-n16.txt};
            \addplot[thick,tab3,scatter,point meta=explicit symbolic,scatter/classes={0={mark=-},1={mark=*}}] table[x=k,y=residual,meta=damping] {results/test3-n32.txt};
            \addplot[thick,tab4,scatter,point meta=explicit symbolic,scatter/classes={0={mark=-},1={mark=*}}] table[x=k,y=residual,meta=damping] {results/test3-n64.txt};
            \addplot[thick,tab5,scatter,point meta=explicit symbolic,scatter/classes={0={mark=-},1={mark=*}}] table[x=k,y=residual,meta=damping] {results/test3-n128.txt};
            \addplot[thick,tab6,scatter,point meta=explicit symbolic,scatter/classes={0={mark=-},1={mark=*}}] table[x=k,y=residual,meta=damping] {results/test3-n256.txt};
            \addplot[thick,tab7,scatter,point meta=explicit symbolic,scatter/classes={0={mark=-},1={mark=*}}] table[x=k,y=residual,meta=damping] {results/test3-n512.txt};
        \end{semilogyaxis}
    \end{tikzpicture}%

    \begin{tikzpicture}[baseline=(current bounding box.center)]
        \begin{semilogyaxis}[width=150pt,height=130pt,mark size=1pt,xmin=-1,xmax=18,ymin=1e-15,ymax=100,title=Test case~4]
            \addplot[thick,dotted,domain=-100:100,update limits=false] {1e-10};
            \addplot[thick,tab1,scatter,point meta=explicit symbolic,scatter/classes={0={mark=-},1={mark=*}}] table[x=k,y=residual,meta=damping] {results/test4-n8.txt};
            \addplot[thick,tab2,scatter,point meta=explicit symbolic,scatter/classes={0={mark=-},1={mark=*}}] table[x=k,y=residual,meta=damping] {results/test4-n16.txt};
            \addplot[thick,tab3,scatter,point meta=explicit symbolic,scatter/classes={0={mark=-},1={mark=*}}] table[x=k,y=residual,meta=damping] {results/test4-n32.txt};
            \addplot[thick,tab4,scatter,point meta=explicit symbolic,scatter/classes={0={mark=-},1={mark=*}}] table[x=k,y=residual,meta=damping] {results/test4-n64.txt};
            \addplot[thick,tab5,scatter,point meta=explicit symbolic,scatter/classes={0={mark=-},1={mark=*}}] table[x=k,y=residual,meta=damping] {results/test4-n128.txt};
            \addplot[thick,tab6,scatter,point meta=explicit symbolic,scatter/classes={0={mark=-},1={mark=*}}] table[x=k,y=residual,meta=damping] {results/test4-n256.txt};
            \addplot[thick,tab7,scatter,point meta=explicit symbolic,scatter/classes={0={mark=-},1={mark=*}}] table[x=k,y=residual,meta=damping] {results/test4-n512.txt};
        \end{semilogyaxis}
    \end{tikzpicture}%
    \begin{tikzpicture}[baseline=(current bounding box.center)]
        \begin{semilogyaxis}[width=260pt,height=130pt,mark size=1pt,xmin=-1,xmax=37,ymin=1e-15,ymax=100,yticklabel=\empty,title=Test case~5]
            \addplot[thick,dotted,domain=-100:100,update limits=false] {1e-10};
            \addplot[thick,tab1,scatter,point meta=explicit symbolic,scatter/classes={0={mark=-},1={mark=*}}] table[x=k,y=residual,meta=damping] {results/test5-n8.txt};
            \addplot[thick,tab2,scatter,point meta=explicit symbolic,scatter/classes={0={mark=-},1={mark=*}}] table[x=k,y=residual,meta=damping] {results/test5-n16.txt};
            \addplot[thick,tab3,scatter,point meta=explicit symbolic,scatter/classes={0={mark=-},1={mark=*}}] table[x=k,y=residual,meta=damping] {results/test5-n32.txt};
            \addplot[thick,tab4,scatter,point meta=explicit symbolic,scatter/classes={0={mark=-},1={mark=*}}] table[x=k,y=residual,meta=damping] {results/test5-n64.txt};
            \addplot[thick,tab5,scatter,point meta=explicit symbolic,scatter/classes={0={mark=-},1={mark=*}}] table[x=k,y=residual,meta=damping] {results/test5-n128.txt};
            \addplot[thick,tab6,scatter,point meta=explicit symbolic,scatter/classes={0={mark=-},1={mark=*}}] table[x=k,y=residual,meta=damping] {results/test5-n256.txt};
            \addplot[thick,tab7,scatter,point meta=explicit symbolic,scatter/classes={0={mark=-},1={mark=*}}] table[x=k,y=residual,meta=damping] {results/test5-n512.txt};
        \end{semilogyaxis}
    \end{tikzpicture}%

    \begin{tikzpicture}[baseline=(current bounding box.center)]
        \begin{axis}[mark size=1pt,xmin=0,xmax=1,ymin=0,ymax=1,hide axis,legend columns=4,legend cell align=left]
            \addlegendimage{thick,tab1,mark=|}
            \addlegendentry{$n=8$}
            \addlegendimage{thick,tab2,mark=|}
            \addlegendentry{$n=16$}
            \addlegendimage{thick,tab3,mark=|}
            \addlegendentry{$n=32$}
            \addlegendimage{thick,tab4,mark=|}
            \addlegendentry{$n=64$}
            \addlegendimage{thick,tab5,mark=|}
            \addlegendentry{$n=128$}
            \addlegendimage{thick,tab6,mark=|}
            \addlegendentry{$n=256$}
            \addlegendimage{thick,tab7,mark=|}
            \addlegendentry{$n=512$}
        \end{axis}
    \end{tikzpicture}%
    \caption{Residuals $\|\nabla E_\nu(\Phi^{(k)})\| / \|(\nu(\{y_i\}))_{1 \leq i \leq N}\|$ in \cref{algo:damped_newton}, with respect to the iteration number $k$, for different values of the discretization parameter $n$. The dotted line corresponds to the tolerance $\varepsilon = 10^{-10}$. Iterations $k$ at which damping occurs, i.e., for which $\tau^{(k)} \neq 1$, are indicated by round marks.}
    \label{fig:residuals}
\end{figure}
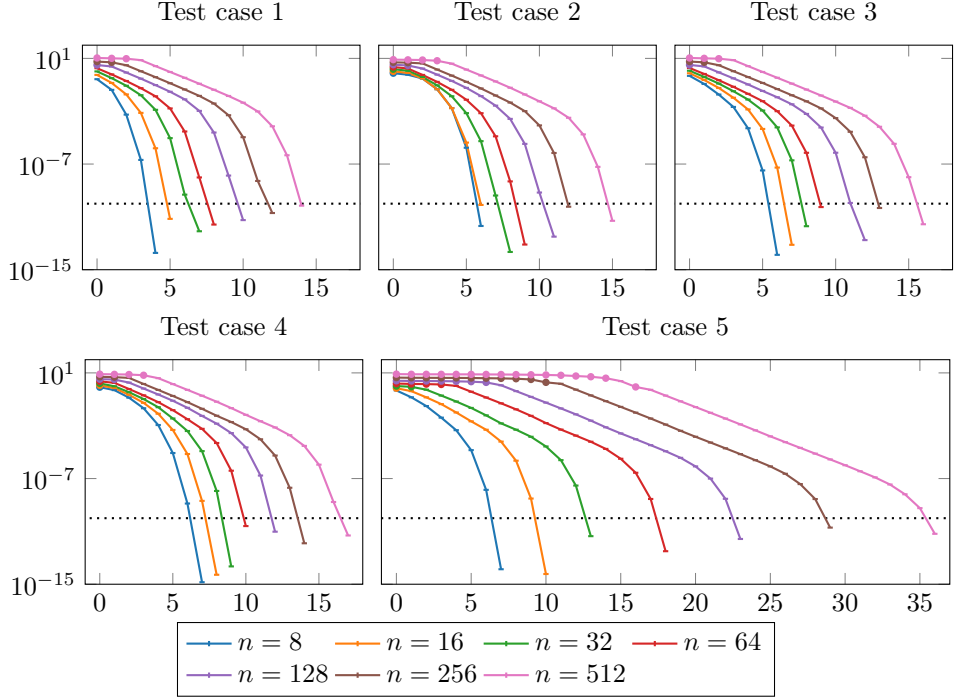

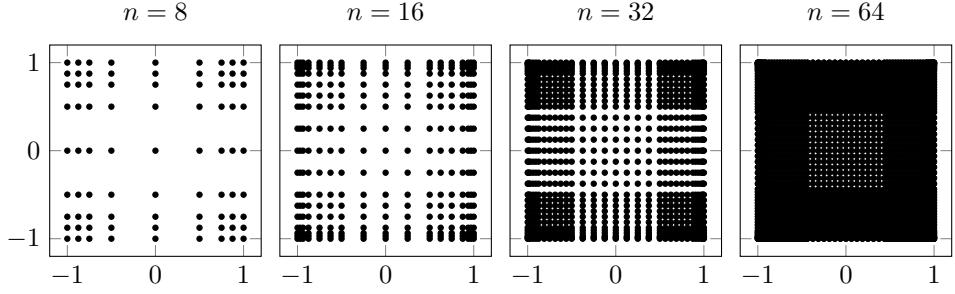
\begin{figure}[t]
    \centering
    \begin{tikzpicture}[baseline=(current bounding box.center)]
        \begin{axis}[width=125pt,height=125pt,mark size=1pt,axis equal,title={$n=8$}]
            \addplot[only marks] table {results/test5-grid-n8.txt};
        \end{axis}
    \end{tikzpicture}%
    \begin{tikzpicture}[baseline=(current bounding box.center)]
        \begin{axis}[width=125pt,height=125pt,mark size=1pt,axis equal,yticklabel=\empty,title={$n=16$}]
            \addplot[only marks] table {results/test5-grid-n16.txt};
        \end{axis}
    \end{tikzpicture}%
    \begin{tikzpicture}[baseline=(current bounding box.center)]
        \begin{axis}[width=125pt,height=125pt,mark size=1pt,axis equal,yticklabel=\empty,title={$n=32$}]
            \addplot[only marks] table {results/test5-grid-n32.txt};
        \end{axis}
    \end{tikzpicture}%
    \begin{tikzpicture}[baseline=(current bounding box.center)]
        \begin{axis}[width=125pt,height=125pt,mark size=1pt,axis equal,yticklabel=\empty,title={$n=64$}]
            \addplot[only marks] table {results/test5-grid-n64.txt};
        \end{axis}
    \end{tikzpicture}%
    \caption{Points of $\supp\,\nu$ in test case~5, for some values of the discretization parameter $n$.}
    \label{fig:grids}
\end{figure}

\paragraph{Test case~5}

We consider the measure $\mu$ and the functions $\psi_\mu$ and $\varphi_\mu$ from test case~1, and a choice of measure $\nu$ that is designed to be adapted to the regularity of the function $\varphi_\mu$, that we know a priori. More precisely, let $n \in 2 \NN^*$ be a discretization parameter as in the previous test cases. For $t \in [-1, 1]$, let $u(t) := (1+t) \log (1+t) + (1-t) \log (1-t)$, so that $\varphi_\mu(y) = u(y_1) + u(y_2) - 4 \log 2$. We compute an $(n+1)$-point discretization of the interval $[-1, 1]$ by first considering the points $-1$ and $1$ and then iteratively bisecting a discretization interval $[t_{j_1}, t_{j_2}]$ that maximizes the $L^\infty([t_{j_1}, t_{j_2}])$ error between $u$ and the affine interpolation of $(t_{j_1}, u(t_{j_1}))$ and $(t_{j_2}, u(t_{j_2}))$. We denote by $-1 = t_{-n/2} < \cdots < t_{n/2} = 1$ the points thus obtained. For convenience, we also choose some arbitrary $t_{-n/2-1} < t_{-n/2}$ and $t_{n/2+1} > t_n$ (whose exact value will have no effect on $\nu$). We let
\begin{equation*}
    \nu := \sum_{j_1=-\frac{n}{2}}^{\frac{n}{2}} \sum_{j_2=-\frac{n}{2}}^{\frac{n}{2}} \left(\int_{\RR^d} u_{j_1, j_2}(y)\, d \mu(y)\right) \delta_{(t_{j_1}, t_{j_2})},
\end{equation*}
where, for any pair $(j_1, j_2) \in \{-n/2, \ldots, n/2\}^2$ and any $y \in \RR^d$,
\begin{align*}
    u_{j_1, j_2}(y) &:= \max \bigg\{0, 1 - \max \bigg\{\frac{y_1 - t_{j_1}}{t_{j_1+1} - t_{j_1}}, \frac{t_{j_1} - y_1}{t_{j_1} - t_{j_1-1}}, \frac{y_2 - t_{j_2}}{t_{j_2+1} - t_{j_2}}, \frac{t_{j_2} - y_2}{t_{j_2} - t_{j_2-1}}, \\
    &\qquad \qquad \qquad \qquad \qquad \frac{y_1 - t_{j_1}}{t_{j_1+1} - t_{j_1}} + \frac{y_2 - t_{j_2}}{t_{j_2+1} - t_{j_2}}, \frac{t_{j_1} - y_1}{t_{j_1} - t_{j_1-1}} + \frac{t_{j_2} - y_2}{t_{j_2} - t_{j_2-1}}\bigg\}\bigg\}
\end{align*}
(compare with equation \cref{eq:basis_function}). The points of $\supp\,\nu$ are displayed in \cref{fig:grids}, for some values of the discretization parameter $n$.

The approach we use for constructing the adapted measure $\nu$ relies on the tensor product structure of $[-1, 1]^2$. Therefore, it cannot be directly extended to the problem of test cases~2 and~4. It also relies on the knowledge of the exact solution $\varphi_\mu$, and thus cannot be applied to practical cases for which the exact solution is unknown. Rather, the numerical results indicate that it may be of interest in future work to study the possibility of choosing $\nu$ adaptively without knowledge of the exact solution.

\paragraph{Error norms}

The error $\|e^{-\psi_\mu(\,\cdot\,)} - e^{-\psi_\nu(\,\cdot\,)}\|_{L^1(\RR^d)}$ from \cref{thm:stability} is not easily computable. Rather, we compute the errors $\|\psi_\mu - \psi_\nu\|_{L^\infty(\RR^d)}$, $\|\varphi_\mu - \varphi_\nu\|_{L^2(\nu)}$, and $\|\varphi_\mu - \varphi_\nu\|_{L^1(\nu)}$.

The errors $\|\varphi_\mu - \varphi_\nu\|_{L^2(\nu)}$, and $\|\varphi_\mu - \varphi_\nu\|_{L^1(\nu)}$ are easily computable, since $\nu$ is finitely supported. The error $\|\psi_\mu - \psi_\nu\|_{L^\infty(\RR^d)}$ could be infinite for general $\mu$ and $\nu$, but it is finite in all our test cases thanks to the fact that we always consider measures $\mu$ and $\nu$ for which the convex hull $\Conv(\supp\,\nu)$ coincides with $\supp\,\mu$. In this case, one can show that
\begin{align}
    \label{eq:legendre_isometry}
    \|\psi_\mu - \psi_\nu\|_{L^\infty(\RR^d)}
    &= \|\varphi_\mu - \varphi_\nu\|_{L^\infty(\supp\,\mu)} \\
    \notag
    &= \max \left\{\max_{x \in \RR^d} (\psi_\mu(x) - \psi_\nu(x)), \max_{y \in \supp\,\mu} (\varphi_\mu(y) - \varphi_\nu(y))\right\}.
\end{align}
Moreover, the maxima in $x$ and $y$ are reached at points corresponding to vertices of the respective epigraphs of $\psi_\nu$ and $\varphi_\nu$ (which are unbounded polyhedra), which makes $\|\psi_\mu - \psi_\nu\|_{L^\infty(\RR^d)}$ computable.

To compare numerical results with \cref{thm:stability}, one can rely on the following observations. Letting $R$ and $r$ be as in \cref{thm:stability}, by \cref{prop:growth_estimate}, there exist $v_1$, $v_2 \in \RR^d$ and a constant $C_d > 0$ depending only on $d$ such that, for all $x \in \RR^d$,
\begin{align}
    \label{eq:numerical_experiments_growth_estimate}
    \psi_\mu(x - v_1) &\geq \frac{r}{2 \sqrt{d}} \|x\| - C_d - d \log R, &
    \psi_\nu(x - v_2) &\geq \frac{r}{2 \sqrt{d}} \|x\| - C_d - d \log R.
\end{align}
Thus, for any $L > 0$,
\begin{align*}
    \|e^{-\psi_\mu(\,\cdot\,)} - e^{-\psi_\nu(\,\cdot\,)}\|_{L^1(\RR^d)}
    &= \|e^{-\psi_\mu(\,\cdot\,)} - e^{-\psi_\nu(\,\cdot\,)}\|_{L^1(B(v_1, L))} \\
    &\quad + \|e^{-\psi_\mu(\,\cdot\,)} - e^{-\psi_\nu(\,\cdot\,)}\|_{L^1(B(v_2, L))} \\
    &\quad + \|e^{-\psi_\mu(\,\cdot\,)} - e^{-\psi_\nu(\,\cdot\,)}\|_{L^1(\RR^d \setminus (B(v_1, L) \cup B(v_2, L)))} \\
    &\lesssim_{d, R/r} L^d \|e^{-\psi_\mu(\,\cdot\,)} - e^{-\psi_\nu(\,\cdot\,)}\|_{L^\infty(\RR^d)} \\
    &\qquad \quad + (1 + (rL)^{d-1}) e^{-rL/(2\sqrt{d})},
\end{align*}
where, using again \cref{eq:numerical_experiments_growth_estimate}, $\|e^{-\psi_\mu(\,\cdot\,)} - e^{-\psi_\nu(\,\cdot\,)}\|_{L^\infty(\RR^d)} \lesssim_d R^d \|\psi_\mu - \psi_\nu\|_{L^\infty(\RR^d)}$.
Choosing $L$ appropriately, one obtains the estimate
\begin{equation}
    \label{eq:error_norms_comparison}
    \|e^{-\psi_\mu(\,\cdot\,)} - e^{-\psi_\nu(\,\cdot\,)}\|_{L^1(\RR^d)} \lesssim_{d, R/r} \|\psi_\mu - \psi_\nu\|_{L^\infty(\RR^d)} \log^d \left(2 + \|\psi_\mu - \psi_\nu\|_{L^\infty(\RR^d)}^{-1}\right).
\end{equation}

\paragraph{Results and comments}

The errors between $\psi_\mu$, $\varphi_\mu$ and $\psi_\nu$, $\varphi_\nu$ in all test cases and for some values of the discretization parameter $n$ are displayed in \cref{fig:error}.

In test cases~1 and~2, all errors experimentally scale like $N^{-1/2}$, i.e., like $W_1(\mu, \nu)$. This is better than the rate $W_1(\mu, \nu)^{1/2}$ that could have been expected in view of \cref{thm:stability}, and of the estimate \cref{eq:error_norms_comparison}. This result is not surprising, as we do not know if the exponent in \cref{thm:stability} is optimal.

In test cases~3 and~4, the modified rule for choosing $\nu$ leads to experimentally larger rates of convergence in the $L^2(\nu)$ and $L^1(\nu)$ norms. This highlights the importance of the choice of $\nu$, beyond what is reflected by the Wasserstein distance $W_1(\mu, \nu)$, which scales like $N^{-1/2}$ in all first fourth test cases. We however do not currently know any theoretical guarantee that following the approach used in test cases~3 and~4 for choosing the measure $\nu$ always leads to such improved rates of convergence. The rate of convergence in the $L^\infty$ norm is experimentally not improved, although the asymptotic constant seems smaller than in test cases~1 and~2.

In test case~5, all errors scale like $N^{-1}$. This improvement of the rates of convergence is consistent with the fact that $\supp\,\nu$ has been chosen to be adapted to the regularity of $\varphi_\mu$.

The residuals at each iteration of \cref{algo:damped_newton} in all experiments are displayed in \cref{fig:residuals}. The results in test cases~1,~2,~3, and~4 are similar, and illustrate superlinear convergence of the Newton method, up to some artifacts in the last iterations that may be related to numerical precision issues. Damping experimentally only occurs in the first few iterations of \cref{algo:damped_newton}. In test case~5, superlinear convergence is still observed, but more iterations are required. Specifically, the number of iterations at which damping occurs is larger, and grows faster with respect to the discretization parameter $n$. This could be related to the fact that, for a fixed value of $n$, the quantity $\min_{y \in \supp\,\nu} \nu(\{y\})$ is significantly smaller in test case~5 than in the other cases.

\section*{Acknowledgments}

G.~B.\ thanks Quentin Mérigot for interesting discussion on the topic of this work.

\iftoggle{author}{
    \bibliographystyle{plain}
    \bibliography{references}
}{
    \bibliographystyle{siamplain}
    \bibliography{references}
}
\end{document}